\documentclass[10pt,reqnol]{amsart}
\usepackage{amsfonts}
\usepackage{url,hyperref,multirow}
\usepackage{color}
\usepackage[dvipsnames]{xcolor}
\usepackage{cases}
\usepackage{subfigure}
\usepackage[ruled,vlined]{algorithm2e}
\usepackage{algorithmic}
\usepackage{cancel}

\usepackage{epsfig,amsmath,bm,caption2,epstopdf}
\usepackage{amssymb,version,graphicx,fancybox,mathrsfs,pifont,booktabs,mathtools}
\usepackage[marginparwidth=2.5cm, marginparsep=3mm]{geometry}

\catcode`\@=11 \theoremstyle{plain}
\@addtoreset{equation}{section}   
\renewcommand{\theequation}{\arabic{section}.\arabic{equation}}
\@addtoreset{figure}{section}
\renewcommand\thefigure{\thesection.\@arabic\c@figure}
\@addtoreset{table}{section}
\renewcommand\thetable{\thesection.\@arabic\c@table}

\newtheorem{thm}{\bf Theorem}[section]
\newtheorem{cor}{\bf Corollary}[section]

\newtheorem{lmm}{\bf Lemma}[section]

\newenvironment{lemma}{\begin{lmm}}{\end{lmm}}
\theoremstyle{remark}
\newtheorem{remark}{\bf Remark}[section]
\theoremstyle{definition}

\newtheorem{exm}{\bf Example} 

\def \ri {{\rm i}}
\def \rd {{\rm d}}

\newcommand{\bs}[1]{\boldsymbol{#1}}

\begin{document}
\baselineskip 14pt
\bibliographystyle{plain}

\title[Logarithmic Nonlinearity] {Unconditional and optimal error analysis of two linearized finite difference schemes for the logarithmic Schr\"odinger equation\thanks{Submitted to the editors DATE.
\funding{The research of the second author is supported by Scientific Research Fund
of Zhejiang Provincial Education Department (Y202352579). }}}
\author[
     T. Wang\, \&\,J. Yan
	]{  Tingchun Wang${}^{\dag}$\;  and\; Jingye Yan${}^{\ddag}$
		}
	\thanks{${}^{\dag}$ School of Mathematics and Statistics, Nanjing University of Information Science and Technology, Nanjing 210044, China. Email: wangtingchun2010@gmail.com.\\
		\indent ${}^{\ddag}$Corresponding author. College of Mathematics and Physics, Wenzhou University, Wenzhou, 325035, China.The research of the first author is supported by Zhejiang Provincial Natural Science Foundation of China (Q24A010031).  Email: yanjingye0205@163.com.
		}

\begin{abstract} In this paper, we propose two linearized finite difference schemes for solving the logarithmic Schr\"odinger equation (LogSE) without the need for regularization of the logarithmic term. These two schemes employ the first-order and the second-order backward difference formula, respectively, for temporal discretization of the LogSE, while using the second-order central finite difference method for spatial discretization. We overcome the singularity posed by the logarithmic nonlinearity $f(u)=u\ln|u|$ in  establishing optimal $l^{2}$-error estimates for the first-order scheme, and an almost optimal $l^{2}$-error estimate for the second-order scheme. Compared to the error estimates of the LogSE in the literature, our error bounds not only greatly improve the convergence rate but also get rid of the time step restriction.  Furthermore, without enhancing the regularity of the exact solution or imposing any requirements on the grid ratio, we establish error estimates of the two proposed schemes in the discrete $H^{1}$ norm. However, the existing results available in the literature either fail  to provide $H^{1}$  error estimates or require certain restrictions on the grid ratio. Numerical results are reported to confirm our error estimates and demonstrate rich dynamics of the LogSE.

\end{abstract}
 \maketitle

\vspace*{-10pt}
\section{Introduction}
In this paper, we consider the following initial-boundary value problem of the logarithmic Schr\"odinger equation (LogSE):
\begin{equation}\label{lognls0B}
\begin{cases}
\ri  \partial_t u(\boldsymbol{x},t)+ \Delta u(\boldsymbol{x},t)=\lambda u (\boldsymbol{x},t)\ln(|u(\boldsymbol{x},t)|^2), \quad   \boldsymbol{x}\in  \Omega, \;\; t>0,\\[2pt]
u(\boldsymbol{x},t)=0,\quad  \boldsymbol{x} \in {\partial \Omega},\;\; t\ge 0; 
\quad u(\boldsymbol{x},0)=u_0(\boldsymbol{x}),\quad \boldsymbol{x}\in \bar \Omega,
\end{cases}
\end{equation}
where $\boldsymbol{x}$, $t$ represent the spatial and temporal coordinates, respectively, and $\lambda \in \mathbb{R}\setminus\{0\}$ measures the force of the nonlinear interaction,
 and $\ri=\sqrt{-1},$  $\Omega\subset\mathbb{R}^{d}$ $(d=1,2,3 )$ is a bounded domain with a smooth boundary, and $u_0$ is a given complex-valued function with a suitable regularity.
The LogSE admits applications to quantum mechanics  \cite{Bialynicki1976nonlinear,Bialynicki1979Gaussons}, quantum optics \cite{Buljan2003Incoherent,krolikowski2000unified}, nuclear physics \cite{Hefter1985Application}, transport and diffusion phenomena \cite{Hansson2009propagation}, quantum gravity \cite{Zloshchastiev2010Logarithmic}.

Like the cubic Schr\"odinger equation which is a closely related model to the LogSE, the initial-boundary value problem \eqref{lognls0B} also preserves the total mass
\begin{align}\label{lognlsmass}
M(t):&=\|u(\cdot, t)\|_{L^2(\Omega)}^2=\int_{\Omega}|u(\boldsymbol{x}, t)|^2 d \boldsymbol{x} \equiv \int_{\Omega}|u_0(\boldsymbol{x})|^2 d \boldsymbol{x}=M(0),
\end{align}
and total energy
\begin{align}\label{lognlsenergy}
E(t): & =\int_{\Omega}|\nabla u(\bs x, t)|^2 d \bs x+(F(|u(\cdot, t)|^2),1)
  \equiv \int_{\Omega}|\nabla u_0(\bs x)|^2 d \bs x+ (F(|u_0|^2),1)=E(0),
\end{align}
where
\begin{equation}\label{F}
(F(|u(\cdot, t)|^2),1):=\int_{\Omega}F(|u(\bs x, t)|^2)d \bs x, \;\;
F(\rho)=\lambda\int_0^\rho \ln (s) d s=\lambda(\rho \ln \rho-\rho), \;\; \rho =|u|^2.
\end{equation}
From a mathematical level, the LogSE \eqref{lognls0B} owns some features which are quite different from the cubic Schr\"odinger equation. First, although the energy of LogSE \eqref{lognls0B} is conserved,
 it does not have a definite sign since the logarithmic term $\ln|u|^2$ can change sign within $\Omega$ as time evolves. Moreover, according to \cite{Carles2018Universal}, whether the nonlinearity is repulsive/attractive (or defocusing/focusing) depends on both $\lambda$ and the value of the density $\rho=|u|^2$. When $\lambda \ln\rho >0$, the nonlinearity $\lambda {u}\ln\rho >0$ is dispersive, and it is attractive when $\lambda \ln\rho <0$. Second, the logarithmic nonlinear term $f(u)=u\ln|u|$ is non-differential at $u=0$.
Besides the first-order derivatives of the nonlinear term with respect to time and space blow up, whenever $u=0$ even for a smooth solution $u.$ In fact, $f(u)$ only is $\alpha$-H\"older continuous with $\alpha\in (0,1)$ (see Lemma \ref{N2L2} below) not Lipschitz continuous. Therefore, solving the Cauchy problem and constructing numerical schemes for LogSE \eqref{lognls0B} are not trivial.

Bao et al  \cite{Bao2019Error} first proposed the regularization of the logarithmic nonlinear term to avoid the blowup of $\ln|u|$ as $|u|\to 0,$ leading to the  regularized logarithmic Schr\"{o}dinger equation (RLogSE):
\begin{equation}\label{Rlognls0C}
\begin{cases}
\ri  \partial_t u^{\varepsilon}(\boldsymbol{x},t)+ \Delta u^{\varepsilon}(\boldsymbol{x},t)=\lambda u^{\varepsilon} (\boldsymbol{x},t)\ln(\varepsilon+|u^{\varepsilon}(\boldsymbol{x},t)|)^2, \quad   \boldsymbol{x}\in  \Omega, \;\; t>0,\\[2pt]
u^{\varepsilon}(\boldsymbol{x},t)=0,\quad  \boldsymbol{x} \in {\partial \Omega},\;\; t\ge 0; 
\quad u^{\varepsilon}(\boldsymbol{x},0)=u_0(\boldsymbol{x}),\quad \boldsymbol{x}\in \bar \Omega,
\end{cases}
\end{equation}
where the regularization parameter   $0<\varepsilon\ll 1.$
Then Bao et al \cite{Bao2019Error} applied the semi-implicit finite difference method in time
and central difference method in space to discretize the RLogSE \eqref{Rlognls0C} and establish its error bound $\mathcal{O}(e^{CT|\ln\varepsilon|^2}(\tau^2+h^2))$ in the discrete $L^{\infty}_t(L^2_x)$ norm.
Later, the first-order Lie-Trotter splitting and Fourier spectral method were considered for the RLogSE
 \eqref{Rlognls0C} in \cite{Bao2019Regularized}, where the conservation of mass is preserved and the constraints for discrete parameters in \cite{Bao2019Error} could be relaxed. And the convergence analysis
arrives at an $\mathcal{O}(|\ln\varepsilon|\tau^{1/2})$ or $\mathcal{O}(\varepsilon^{-1}\tau)$ error bound in the discrete $L^{\infty}_t(L^2_x)$ norm. In \cite{Carles2022numerical} Carles and Su gave the results $\mathcal{O}(|\ln\varepsilon|\tau^{1/2})$ concerning the error estimates of the first-order regularized Lie-Trotter splitting method by solving the LogSE with a harmonic potential. Moreover, Bao et al \cite{Bao2022Error} solved LogSE \eqref{lognls0B} from the energy regularization level different from  \eqref{Rlognls0C} in their first work \cite{Bao2019Error}. In this paper, they obtained
theoretically the convergence rate at $\mathcal{O}(|\ln\varepsilon|\tau^{1/2})$ for the Lie-Trotter scheme or  $\mathcal{O}(\varepsilon^{-1}\tau)$,  and convergence rate at $\mathcal{O}(\varepsilon^{-3}\tau^2)$ for the  Strang splitting scheme. Numerical results confirm the error bounds and indicate local energy regularization performs better than directly regularizing the singular nonlinearity globally. In \cite{Qian2023Novle}, Qian et al constructed high-order mass- and energy-conservative Runge-Kutta
integrators for the logarithmic Schr\"{o}dinger equation with the same regularized term with \cite{Bao2019Error}. Based on these ideas,  four regularized finite difference methods for solving the logarithmic Klein-Gordon equation were constructed in \cite{Yan2021Two,Yan2022A,Yan2021R} and it was shown that these error bounds reach at $\mathcal{O}(\varepsilon+e^{CT|\ln\varepsilon|^2}(\tau^2+h^2))$ and $\mathcal{O}(\varepsilon+e^{\frac{T}{2\varepsilon}}({\tau^2}{\varepsilon^{-2}}+h^2))$, respectively. The error analysis in the above references indicates that the existing numerical schemes of LogSE either require certain restrictions on the grid ratio or suffer from loss of convergence order due to logarithmic nonlinearity.

Recently, there has been growing interest in constructing numerical schemes for the non-regularized logarithmic PDE. With the understanding that $f(0)=0$ when $u=0$, then the nonlinearity is continuous though $f'(u)$ blows up at $u=0$. Some works were carried out to successfully discretize  the original LogSE \eqref{lognls0B} without regularizing the nonlinear term. Wang et al \cite{Wang2024Error} proposed a linearized implicit-explicit time-discretisation with finite element approximation in space for the LogSE without regularization and established the nearly optimal error estimate $\mathcal{O}(\tau+h^{r+1}|\ln h|)$ in the discrete $L^{2}$ norm. However,  their convergence result requires a constraint on the grid ratio, yet no $H^{1}$ error estimate was given. Paraschis and Zouraris \cite{Paraschis2023On}  studied  the  nonlinear implicit Crank-Nicolson-type  scheme with a small and nonnegative parameter $\varepsilon$  to solve the LogSE \eqref{lognls0B} in two dimensions, they first gave a nearly optimal estimate of the local truncation error and then established, with or without using $f_\varepsilon(s):=s\ln(\varepsilon+|s|)$ as a regularization of $f(s)$,  a corresponding nearly optimal estimate of the global error with order $\mathcal{O}(\tau^2+\tau^2|\ln (\tau)|+h_1^2+h_2^2)$ for $\varepsilon=0$ and $\mathcal{O}(\tau^2+\tau^2|\ln (\varepsilon)|+h_1^2+h_2^2)$ for $\varepsilon>0$ in the discrete $L^{\infty}_t(L^2_x)$ norm.
 In \cite{zhang2024low}, Zhang and Wang devoted to the analysis of Lie-Trotter time-splitting and Fourier spectral methods of the LogSE without regularisation and in low regularity setting. They showed that, under the regularity  $u_0 \in L^2(\mathbb{T}^d)$ and $u \in C((0, T] ; H^s(\mathbb{T}^d) \cap L^{\infty}(\mathbb{T}^d))$ for some $0<s<1$, the $L^2$ error is of order $\mathcal{O}((\tau^{s / 2}+N^{-s}) \ln N)$. Recently, Bao, Ma and Wang \cite{bao2024optimal} presented an exponential wave integrator Fourier spectral method for the LogSE and established the $L^2$ error bound with order $\mathcal{O}(\tau|\ln\tau|^2+h^{2}|\ln h|)$ under a CFL-type time step size restriction $\tau|\ln\tau|\le h^2/|\ln h|$.

Finite difference methods are designed with simplicity, making them easy to program and highly popular for applications in practical problems. Therefore this paper is devoted to developing and analyzing two efficient finite difference schemes for approximating the solution to the LogSE \eqref{lognls0B}, without applying regularization to the nonlinear term. These two schemes are introduced utilizing the first-order and second-order backward difference formulas, respectively, for the temporal discretization of the LogSE, while employing the second-order central finite difference method for spatial discretization. Moreover, each scheme employs a different linearized discrete method for the nonlinear term of the LogSE. For simplicity, we use  BDF1  and BDF2   to denote the first-order and second-order schemes, respectively. The highlights of this paper are outlined as follows:
\begin{itemize}
  \item These two schemes are linearly implicit, which means that only a system of linear algebraic equations needs to be solved at each time level. This can be quickly achieved using the FFT-type solver. Therefore, they are very efficient in practical computations.
   \vskip 0.1cm
  \item The BDF1 scheme exhibits unconditional convergence and optimal error estimate of $\mathcal{O}(\tau + h^{2} )$ in the discrete $L^{2}$ norm, where $\tau$ and $h$ are the mesh size and time step, respectively;
       \vskip0.1cm
    \item The BDF2 scheme exhibits unconditional convergence and almost optimal error estimate, the convergence order in the discrete $L^{2}$ norm is of $\mathcal{O}(h^{2}+\tau^{2}{\ln|\tau|})$;
       \vskip 0.1cm
    \item The error estimates of the two BDF schemes are also established in the discrete $H^{1}$ norm without enhancing the regularity of the exact solution or imposing any requirements for grid ratio.
\end{itemize}

The rest of this paper is organized as follows. In Section  \ref{sect2}, we propose two efficient finite difference schemes  for solving the LogSE \eqref{lognls0B}, and then discuss the solvability of them. In Section  \ref{sect3}, we establish the unconditional and optimal (or nearly optimal) error estimates for these numerical schemes. The final section is dedicated to numerical results and discussions.

Throughout the paper, we denote $C$ as a generic positive constant
which may have different values at different circumstances but are
independent of the discrete parameters, i.e. the time step $\tau$
and mesh size $h$, and we adopt the notation $w\lesssim{v}$ to
denote $|w|\leq{Cv}$.

\medskip
\section{Finite difference schemes }\label{sect2}
In this section, we introduce two finite difference time-domain methods for the LogSE equation \eqref{lognls0B} in $d$ dimensions on a bounded computational domain $\Omega\in\mathbb{R}^{d}$ ($d=1,2,3$) equipped with homogeneous Dirichlet boundary conditions and state our main error estimate results.
\subsection{Numerical methods}\label{fullDis}
To simplify the notation, we only demonstrate the methods in a two-dimensional context, meaning that $d=2$, and $\Omega$, which represents a rectangle in $\mathbb{R}^{2}$, is given by $(a,b)\times(c,d)$. The initial-boundary value problem \eqref{lognls0B} in two dimensions reads
\begin{equation}\label{lognls0B-2d}
\begin{cases}
\ri  \partial_t u(x,y,t)+ \Delta u(x,y,t)=\lambda u (x,y,t)\ln(|u(x,y,t)|^2),\;\;
 (x,y)\in \Omega, \;\; t>0,\\[2pt]
 u(x,y,t)=0,\;\;(x,y)\in \partial\Omega,\;\; t\ge 0;  \\[2pt]
 u(x,y,0)=u_0(x,y),\;\; (x,y)\in \overline{\Omega}.
\end{cases}
\end{equation}
Generalized to the one-dimensional case with $\Omega=(a,b)$ being a interval in $\mathbb{R}$ and three-dimensional case with $\Omega=(a,b)\times (c,d)\times (e,f)$ being a cube in  $\mathbb{R}^{3}$ are straightforward.
Choose the time-step $\tau=T/N$ with $N$ being a positive integer, and denote time steps
$t_{n}=n\tau,\hspace{0.2cm}n=0,1,2,\cdots,N$, where $0<T<T_{\max}$
with $T_{\max}$ being the maximal existing time of the solution;
  let the mesh size $h_x:=\frac{b-a}{J}$ and $h_y:=\frac{d-c}{K}$ with $J,K$ being two positive integers and denote $h:=h_{\max}=\max\{h_x,h_y\},\,h_{\min}:=\min\{h_x,h_y\}$, and  $x_{j}:=a+jh_x$, $y_{k}:=c+kh_y$ for $j=0,1,\ldots,J,\; k=0,1,\ldots,K$. We here assume that $h\lesssim {h}_{\min}$.
Define the index sets as
\begin{equation*}
\begin{split}
\mathcal{T}_{h}&=\big\{(j,k)|j=1,2,\ldots,J-1;\,k=1,2,\ldots,K-1\big\},\\
\mathcal{T}_{h}^{0}&=\big\{(j,k)|j=0,1,2,\ldots,J;\,k=0,1,2,\ldots,K\big\},\quad \Gamma_h=\mathcal{T}_{h}^{0}\backslash \mathcal{T}_{h}.
\end{split}
\end{equation*}
Assume that $u^{n}_{j,k}$ is the approximation to the exact solution $ u(x_{j},y_{k},t_{n})$ for $(j,k)\in \mathcal{T}_{h}^{0}$ and $n\geq 0$. Denote by $ u^{n} \in \mathbb{C}^{(J+1)\times(K+1)} $ the numerical solution vector at time $t=t_{n}$. The followings are the finite difference operators:
\begin{equation*}
\begin{split}
&\delta^{+}_{x}u^{n}_{j,k}=\frac{1}{h_x}({u^{n}_{j+1,k}-u^{n}_{j,k}}),\quad\delta^{+}_{y}u^{n}_{j,k}=\frac{1}{h_y}({u^{n}_{j,k+1}-u^{n}_{j,k}}), \quad \delta_{\nabla}^{+}{u}_{j,k}^{n}=(\delta_{x}^{+} u_{j,k}^{n}, \delta_{y}^{+} u_{j,k}^{n})^{T},\\
&\delta^{2}_{x}u^{n}_{j,k}=\frac{1}{h_x^{2}}({u^{n}_{j+1,k}-2u^{n}_{j,k}+u^{n}_{j-1,k}}), \quad
\delta^{2}_{y}u^{n}_{j,k}=\frac{1}{h_{y}^{2}}({u^{n}_{j,k+1}-2u^{n}_{j,k}+u^{n}_{j,k-1}}),\\
&
 \delta_{\nabla}^{2} u_{j,k}^{n}=\delta_{x}^{2}{u}_{j,k}^{n}+\delta_y^{2}{u}_{j,k}^{n},\quad
\delta^{+}_{t}u^{n}_{j,k}=\frac{1}{\tau}({u^{n+1}_{j,k}-u^{n}_{j,k}}),\quad
\delta^{-}_{t}u^{n+1}_{j,k}=\delta^{+}_{t}u^{n}_{j,k},\\
&D_{t}^{-}u^{n+1}_{j,k}=\frac{1}{2\tau}({3u^{n+1}_{j,k}-4u^{n}_{j,k}+u^{n-1}_{j,k}}),\quad
  u^{n+\frac{1}{2}}_{j,k}=\frac{1}{2}({u^{n}_{j,k} + u^{n+1}_{j,k}}).
\end{split}
\end{equation*}
We introduce the space of grid functions $X_{h}^{0}$ as
\begin{equation*}
X_{h}^0=\big\{ u=(u_{j,k})_{(j,k) \in \mathcal{T}_{h}^{0}}| u_{j,k}=0\;{\rm{when}}\;(j,k) \in \Gamma_{h} \big\}\subseteq \mathbb{C}^{(J+1)\times(K+1)}.
\end{equation*}
The inner product and  discrete $L^{r} (1\le{r}<+\infty)$, semi-$H^{1}$, semi-$H^{2}$, $L^{\infty}$ norms  over $X_{h}^0$ are defined as
\begin{equation*}
\begin{split}
&(u,v)_{h}=h_x h_y \sum \limits^{J-1}_{j=1} \sum \limits^{K-1}_{k=1}u_{j,k} \overline{v}_{j,k},\quad
\langle \delta_x^{+}u, \delta_x^{+}v\rangle_{h}=h_x h_y \sum \limits^{J-1}_{j=0} \sum \limits^{K-1}_{k=0}\delta_x^{+}u_{j,k}\delta_x^{+} \overline{v}_{j,k},\\
&\langle \delta_y^{+}u, \delta_y^{+}v\rangle_{h}=h_x h_y \sum \limits^{J-1}_{j=0} \sum \limits^{K-1}_{k=0}\delta_y^{+}u_{j,k}\delta_y^{+}\overline{v}_{j,k},\quad \|u\|_{r}=(h_x h_y \sum \limits^{J-1}_{j=1} \sum \limits^{K-1}_{k=1}|u_{j,k}|^{r})^{\frac{1}{r}},\\
&\langle\delta_{\nabla}^{+}u,\delta_{\nabla}^{+}v\rangle_{h}:=\langle \delta_x^{+}u, \delta_x^{+}v\rangle_{h}+\langle \delta_y^{+}u, \delta_y^{+}v\rangle_{h},
\quad |u|_{1}:= \langle\delta_{\nabla}^{+}u,\delta_{\nabla}^{+}u\rangle_{h}^{\frac{1}{2}},\\
&
|u|_{2}=(\delta^2_{\nabla}u,\delta^2_{\nabla}u)_{h}^{\frac{1}{2}},\quad
 \|u\|_{{\infty}}=\sup _{(j,k)\in\mathcal{T}_{h}}|u_{j,k}|,
\end{split}
\end{equation*}
where $u,v\in X_{h}^0$ and $\overline{v}$ represents the conjugate of $v$. It is obvious that $\|u\|_{2}=(u,u)_{h}^{\frac{1}{2}}$, and it is easy to verify the following equations:
\begin{equation*}
\begin{split}
&(\delta_{x}^{2}u, v)_{h}=-\langle\delta_{x}^{+}{u}, \delta_{x}^{+}{v}\rangle_{h},\quad
(\delta_y^2 u, v)_{h}=-\langle\delta_{y}^{+}u, \delta_{y}^{+}v\rangle_{h}.
\end{split}
\end{equation*}
To simplify the notation, we denote $\|u\|_{2}$ by $\|u\|$ throughout the paper.

Based on the above preparations, we introduce the following two finite difference time-domain methods for solving the initial-boundary value problem of LogSE \eqref{lognls0B-2d}:
\begin{itemize}
  \item[I.]  BDF1  scheme
  \begin{equation}\label{lognlsbdf1}
    \ri \delta_{t}^{-}u^{n+1}_{j,k}+\delta^2_{\nabla} u^{n+1}_{j,k} =2\lambda f( u^{n}_{j,k}),\;\;(j,k)\in\mathcal{T}_{h},\; n=0,1,\ldots,N-1;
\end{equation}
  \item[II.]  BDF2 scheme
  \begin{subequations}\label{lognlsbdf2}
     \begin{align}
    &\ri D_{t}^{-}{u}^{n+1}_{j,k}+\delta^{2}_{\nabla} u^{n+1}_{j,k}=2\lambda f(2u^{n}_{j,k}-u^{n-1}_{j,k}),\;\;(j,k)\in\mathcal{T}_{h},\; n=1,2,\ldots,N-1,\label{lognlsLEFD-1}\\
    &\ri \delta^{-}_{t}{u}^{1}_{j,k}+\delta^{2}_{\nabla}{u}^{1}_{j,k} =2\lambda f(u^{0}_{j,k}),\;\;(j,k)\in\mathcal{T}_{h}.\label{lognlsLEFD-2}
  \end{align}
\end{subequations}
\end{itemize} 

\subsection{Solvability}
 We now study the existence and uniqueness of the BDF1 scheme and the BDF2 scheme.
\begin{lemma}[Unique solvability of the BDF1 scheme] \label{ExistUnique} For any fixed $\tau,h_x,h_y>0$ and $ 0\le n\le N-1,$ the discretised problem \eqref{lognlsbdf1} has  a unique solution $ {u}^{n+1}  \in X_h^0.$
\end{lemma}
\begin{proof}
From \eqref{lognlsbdf1} we can see that $u^0$ exists, then by using mathematical induction, we assume that $u^n$ is known and try to prove that the solution $u^{n+1}$ exists. In fact, for a given $u^n$, the BDF1 scheme \eqref{lognlsbdf1} just is a nonhomogeneous system of linear algebraic equations with respect to the unknown vector $u^{n+1}\in\mathbb{C}^{(J+1)\times(K+1)}$. Hence, to analyze the unique
solvability of the system of linear algebraic equations \eqref{lognlsbdf1}, we just only prove that the corresponding homogeneous system
\begin{equation}\label{IMEXhomo}
\begin{cases}
\frac{\ri}{\tau}u^{n+1}_{j,k}+\delta^2_{\nabla} u^{n+1}_{j,k} =0,\;\; (j,k) \in \mathcal{T}_{h},\\
u^{n+1}_{j,k}=0, \;\; (j,k) \in \Gamma_{h}
\end{cases}
\end{equation}
has a unique zero solution. To do this, we compute the inner production of \eqref{IMEXhomo} with $u^{n+1}$  and take the imaginary part to obtain $\|u^{n+1}\| =0$, which implies
$u^{n+1}_{j,k} =0$ for  $(j,k) \in \mathcal{T}_{h}^{0}$.
This ends the proof.
\end{proof}
Similarly, we can obtain the existence and uniqueness of the BDF2 scheme.
\begin{lemma}[Unique solvability of the BDF2 scheme] \label{ExistUniqueLEFD} For any fixed $\tau,h_x,h_y>0$ and $ 0\le n\le N-1,$ the discretised problem \eqref{lognlsbdf2} has a unique solution $ {u}^{n+1}  \in X_h^0.$
\end{lemma}

\subsection{Some useful lemmas for error estimates}
\begin{lemma}\cite{Wang2024Error}\label{Holderlmm}
Let $f (z)=z \ln |z|$ for  $z\in \mathbb C.$   If $|u|, |v|\ge  0,$ then we have
\begin{equation}\label{fuvl}
|f(u)-f (v)|\leq  { (|\!\ln y| +1)}\,  |u-v|,\quad y:=\max\{|u|,|v|\}.
\end{equation}
 Given $\alpha\in (0,1),$ denote $\delta_\alpha:=e^{\frac \alpha {\alpha-1}}.$   If  $|u|, |v|\in [0, \epsilon]$ and $0\le \epsilon\le \delta_\alpha,$  then  we have
\begin{equation}\label{holderbnd}
|f(u)-f (v)|
  \le {\mathcal H_\alpha}({\epsilon})  |u-v|^{\alpha},\quad {\mathcal H_\alpha}({\epsilon}):=(2\epsilon)^{1-\alpha}\big(|\!\ln \epsilon|+1\big),
\end{equation}
so  $f (z)$ is {\rm(}locally{\rm)} $\alpha$-H\"older continuous.
\end{lemma}
\begin{lemma}\label{imbedding}
For any grid function $u \in X_h^0$, there holds the following inequality:
\begin{equation}
\|u\|_{2\alpha}\le |\Omega|^{\frac{1}{2\alpha}-\frac{1}{2}}\|u\|,
\end{equation} where $\alpha\in(0,1)$.
\end{lemma}
\begin{proof}
With H\"older's inequality, we have
\begin{equation*}
\begin{split}
\|u\|_{2\alpha}^{2\alpha}&=h_xh_y\sum \limits^{J-1}_{j=1}\sum \limits^{K-1}_{k=1}|u_{j,k}|^{2\alpha}=h_xh_y\sum \limits^{J-1}_{j=1}\sum \limits^{K-1}_{k=1}1\cdot|u_{j,k}|^{2\alpha}\\
&\le (h_xh_y\sum \limits^{J-1}_{j=1}\sum \limits^{K-1}_{k=1}1^{\beta})^{\frac{1}{\beta}}\big(h_xh_y\sum \limits^{J-1}_{j=1}\sum \limits^{K-1}_{k=1}(|u_{j,k}|^{2\alpha})^{\frac{1}{\alpha}}\big)^{\alpha}
 \le |\Omega|^{\frac{1}{\beta}}\|u\|^{2\alpha},
\end{split}
\end{equation*}
where $\frac{1}{\beta}=1-\alpha$. Then we obtain
$ \|u\|_{2\alpha}\le |\Omega|^{\frac{1}{2\alpha}-\frac{1}{2}}\|u\|$,
which completes the proof.
\end{proof}
\begin{lemma}\label{N2L2}
Let $f (u)=u \ln |u|$ be a composite function. Assume that  $u,v\in  {X}_{h}^0 $ and  denote
$\Lambda_\infty= \max\{ \|u \|_{\infty}, \|v \|_{\infty} \}.$
 For any $\alpha \in (0,1)$ and $\epsilon \in (0, \delta_\alpha]$ with   $\delta_\alpha=e^{\frac \alpha {\alpha-1}},$ we have the following bounds.
\begin{itemize}
\item[(i)]If $\Lambda_\infty >\epsilon,$   then
\begin{equation}\label{ful2est}
\begin{split}
\|f(u)-f(v)\|^2
& \le  
 {\mathcal H}^2_\alpha(\epsilon) |\Omega|^{1-\alpha}\|u\|^{2\alpha} +\Upsilon^2(\epsilon, \Lambda_\infty) \, \|u-v\|^{2},
\end{split}
\end{equation}
where
\begin{equation}\label{holderbnd}
{\mathcal H_\alpha}({\epsilon}):=(2\epsilon)^{1-\alpha}\big(|\!\ln \epsilon|+1\big),
\end{equation}
 and
\begin{equation}\label{Hupsilon}
\Upsilon(\epsilon, \Lambda_\infty)= \max_{\epsilon\le y\le  \Lambda_\infty } \!\!\big\{|\!\ln y|+1\big\}.
\end{equation}
\item[(ii)] If $\Lambda_\infty \le \epsilon$, then
\begin{equation}\label{fuvHolder}
\| f(u) - f(v) \| \le \mathcal H _\alpha( \epsilon ) \| u - v \|_{2\alpha}^{\alpha}.
\end{equation}
\end{itemize}
\end{lemma}
\begin{proof}  We first consider $\Lambda_\infty>\epsilon,$ and  decompose  the index sets into $\mathcal{T}_{h}^{0}=\cup_{i=1}^4\mathcal{T}_{i}$, where
\begin{equation}\label{decomOmega}
\begin{aligned}
&\mathcal{T}_{1}=\{(j,k)\in \mathcal{T}_{h}^{0}| |v_{j,k}|\le \epsilon, |u_{j,k}| \le \epsilon \},  &&\mathcal{T}_{2}=\{(j,k)\in \mathcal{T}_{h}^{0}|  |v_{j,k}| \le \epsilon, |u_{j,k}| > \epsilon\}, \\
& \mathcal{T}_{3}=\{(j,k)\in \mathcal{T}_{h}^{0}|  |u_{j,k}| \le \epsilon, |v_{j,k}| > \epsilon \}, \;\; &&\mathcal{T}_{4}=\{(j,k)\in \mathcal{T}_{h}^{0}|   |u_{j,k}|> \epsilon, |v_{j,k}| > \epsilon\}.
\end{aligned}
\end{equation}
From  \eqref{holderbnd}, and when $(j,k)\in \mathcal{T}_{1}$, we have
\begin{equation*}\label{case1}
\begin{split}
h_xh_y\sum_{(j,k)\in \mathcal{T}_{1}}| f(u_{j,k}) - f(v_{j,k}) |^2 & \le h_xh_y \mathcal{H}_\alpha^2(\epsilon)  \sum_{(j,k)\in \mathcal{T}_{1}}  |u_{j,k}-v_{j,k}|^{2\alpha}=
 \mathcal{H}_\alpha^2(\epsilon) \| u - v \|_{ {2\alpha}(\mathcal{T}_{1}) }^{2\alpha}.
\end{split}
\end{equation*}
Moreover, when $(j,k)\in \mathcal{T}_{2}$, we obtain
\begin{equation*}
\begin{aligned}
h_xh_y\sum_{(j,k)\in \mathcal{T}_{2}}|f(u_{j,k})-f(v_{j,k})|^2 & \le  h_xh_y\sum_{(j,k)\in \mathcal{T}_{2}} ( | \!\ln |u_{j,k}|| + 1)^2  |u_{j,k}-v_{j,k}|^2 \\
& \le \sup_{(j,k)\in \mathcal{T}_{2}} \{ |\! \ln |u_{j,k}| | + 1\}^2 h_xh_y\sum_{(j,k)\in \mathcal{T}_{2}} |u_{j,k} - v_{j,k}|^2\\
&\le  \max_{\epsilon \le \vartheta_{j,k} \le \|u\|_{\infty( \mathcal{T}_{2})}}\{|\! \ln \vartheta_{j,k} | + 1\}^2  h_xh_y\sum_{(j,k)\in \mathcal{T}_{2}}  |u_{j,k} - v_{j,k}|^2\\
&\le\big(\max_{\epsilon \le \vartheta_{j,k} \le \Lambda_\infty} \{| \! \ln \vartheta_{j,k} | +1 \}\big)^2  \|u-v\|^2_{( \mathcal{T}_{2})},
\end{aligned}
\end{equation*}
and similarly when $(j,k)\in \mathcal{T}_{3}$, we have
\begin{equation*}
\begin{aligned}
h_xh_y\sum_{(j,k)\in \mathcal{T}_{3}}|f(u_{j,k})-f(v_{j,k})|^2
&\le\big(\max_{\epsilon \le \vartheta_{j,k} \le \Lambda_\infty} \{| \! \ln \vartheta_{j,k} | +1 \}\big)^2  \|u-v\|^2_{(\mathcal{T}_{3})}.
\end{aligned}
\end{equation*}
At last, when $(j,k)\in \mathcal{T}_{4}$, we define $w_{j,k}:= \max\{ |u_{j,k}|, |v_{j,k}|\}$, and note that
\begin{equation*}\label{Omega2}
\epsilon \le w_{j,k} \le  \max\{ \|u\|_{\infty (\mathcal{T}_4)}, \|u\|_{\infty (\mathcal{T}_4)}\}\le \Lambda_\infty,
\end{equation*}
 we derive from \eqref{Hupsilon} that
\begin{equation*}
\begin{split}
h_xh_y\sum_{(j,k)\in \mathcal{T}_{4}}|f(u_{j,k})-f(v_{j,k})|^2 & \le h_xh_y\sum_{(j,k)\in \mathcal{T}_{4}} ( | \!\ln w_{j,k} | + 1)^2  |u_{j,k}-v_{j,k}|^2 \\
&\le   \sup_{(j,k)\in \mathcal{T}_{4}} \{ |\! \ln w_{j,k} | + 1\}^2 h_xh_y\sum_{(j,k)\in \mathcal{T}_{4}}  |u_{j,k}-v_{j,k}|^2 \\
 &\le \big(\max_{\epsilon \le y \le \Lambda_\infty} \{| \! \ln y | +1 \}\big)^2\,  \|u-v\|^2_{(\mathcal{T}_{4})}.
 \end{split}
\end{equation*}
Summing up the above ``local" bounds and noticing Lemma \ref{imbedding} yield
\begin{equation*}
\begin{split}
\|f(u)-f(v)\|^2 &=h_xh_y \sum_{i=1}^4 \sum_{(j,k)\in \mathcal{T}_{i}} |f(u_{j,k}) - f(v_{j,k})|^2\\
&\le     {\mathcal H}^2_\alpha(\epsilon)  \|u-v\|^{2\alpha}_{{2\alpha}} +
\Upsilon^2({\epsilon},\Lambda_\infty)  \, \|u-v\|^{2}\\
&\le {\mathcal H}^2_\alpha(\epsilon)|\Omega|^{1-\alpha}\|u\|^{2\alpha} +
\Upsilon^2({\epsilon},\Lambda_\infty)  \, \|u-v\|^{2},
\end{split}
\end{equation*}
which leads to \eqref{ful2est}. It is clear that if  $\Lambda_\infty \le \epsilon$, then $\mathcal{T}_{h}^{0}=\mathcal{T}_1,$  i.e.,
 $\mathcal{T}_i=\emptyset $ for $i=2,3,4$ in \eqref{decomOmega}. Therefore \eqref{fuvHolder} is a direct consequence  of
\eqref{ful2est}.
\end{proof}
\begin{remark}
It is noteworthy that Wang, Yan and Zhang presented in their work \cite{Wang2024Error} a bound for $\|f(u)-f(v)\|_{L^{2}(\Omega)}$ in the context of the finite element method. However, the bound presented in Lemma \ref{N2L2} differs from this, as it is specifically tailored for the finite difference time-domain  method.
\end{remark}
 In the error analysis of the BDF1 scheme and the BDF2 scheme, we shall use the following Gr\"{o}nwall's inequality in accordance with the $\alpha$-H\"older regularity.
\begin{lemma} \cite{Wang2024Error} \label{disgronwalls}
Let $c_1, c_2, c_3$ be  positive constants and let  $\alpha\in (0,1]$. Suppose that  a sequence $\{y(n)\}$ satisfies
\begin{equation}\label{disgronwalleqs1}
y(n) \leq  {c}_{1} + {c}_{2}\tau\sum^{n-1}_{m=0}\limits {y}^{\alpha}(m) + {c}_{3}\tau\sum^{n-1}_{m=0}\limits y(m),\quad {n}\geq {1},
\end{equation}
then we have
\begin{equation}\label{disgronwalleqs20}
y(n)  \leq c_1 \bigg(1+(c_1^{\alpha-1}c_2+c_3)\frac{(1+\alpha c_1^{\alpha-1}c_2\tau+c_3\tau)^n-1}{\alpha c_1^{\alpha-1}c_2+c_3}\bigg),\quad n\ge 1.
\end{equation}
\end{lemma}
\begin{lemma}\cite{Cazenave1980}\label{LipSC}
For any $u, v \in \mathbb{C}$, we obtain
\begin{equation}\label{LipSCeq}
\big|{\rm{Im}} [(f(u)-f(v))(u-v)^*]\big| \leq |u-v|^{2}.
\end{equation}
\end{lemma}
%
\begin{lemma}[Lemma $2.1$ in \cite{Paraschis2023On}]\label{Nep}
For any $\varepsilon>0$, it holds that
\begin{equation}\label{LipSC-1eq}
\sup_{z\in \mathbb{C}}\big|f(z)-f_\varepsilon(z)\big| \leq \varepsilon,
\end{equation} where $f_{\varepsilon}(z):=z\ln(\varepsilon+|z|).$
\end{lemma}
%
\begin{lemma}[Lemma $2.2$ in \cite{Paraschis2023On}]\label{Nep}
For any $c>e$ and $\varepsilon\in(0,\frac{1}{2c})$, it holds that
\begin{equation}\label{LipSC-1eq-2}
 \big|f_\varepsilon(z)-f_\varepsilon(w)\big| \leq 2|\ln(\varepsilon)||z-w|, \quad
 \forall z,w\in \Gamma_{c},
\end{equation}
where $\Gamma_{c}:=\{v\in\mathbb{C}: |v|\leq{c}\}$.
\end{lemma}
%
\begin{lemma}[Lemma $3.1$ in \cite{Bao2013Optimal}]\label{l4norm}
For any $u \in X_h^0$, we get
\begin{equation}
\|u\|_4 \le \|u\|^{\frac{1}{2}}|u|^{\frac{1}{2}}_1.
\end{equation}
\end{lemma}
\begin{lemma}\cite{zhou1990application}\label{uinflemma}
For any $u \in X_h^0$, we have
\begin{equation}\label{uinf}
\|u\|_{\infty} \leq C\|u\|^{\frac{1}{2}}(|u|_2+\|u\|)^{\frac{1}{2}},
\end{equation}
where $C$ depends on the size of $\Omega$ but is independent of $u$ and $h$.
\end{lemma}
\section{Error estimates}\label{sect3}
Assume that $u$ is smooth enough and satisfies
\begin{equation}\label{ucnA}
u\in C^2([0,T];L^{\infty}(\Omega))\cap C^1([0,T];W^{2,\infty}(\Omega))\cap C^0([0,T];W^{4,\infty}(\Omega)\cap H^1_0(\Omega)).
\end{equation}
Denote the exact grid solution $U^{n}\in X_{h}^{0}$ and error grid function $e^{n}\in X_{h}^{0}$ as
\begin{align}
&U_{j,k}^{n}:=u(x_{j},y_{k},t_{n}), \quad e^{n}_{j,k}:=U_{j,k}^{n}-u_{j,k}^{n}, \quad {(j,k)}\in\mathcal{T}_{h}^{0},\;\; n=0,1,2,\dots,N.\nonumber
\end{align}
\begin{thm}\label{enzhao0} Assume $h\lesssim h_{min}$, then
under the assumption \eqref{ucnA}, there exist $h_0>0,\,\tau_0>0$ sufficiently small, when $0<h\le h_0,\,0<\tau\le \tau_0$,
 we have the following error estimates of the BDF1 scheme \eqref{lognlsbdf1}:
\begin{subequations}\label{errbndiefd}
\begin{align}
&\|e^{n}\| \le  C_{u}  (\tau + h^2),\quad  n=0,1,\cdots, N,\label{errbndiefd-1}\\
&|e^{n}|_{1} \le  C_{u}{\tau}^{-\frac{1}{2}}(\tau + h^2),\quad  n=0,1,\cdots, N,\label{errbndiefd-2}
\end{align}
\end{subequations}
where $C_u$ is a generic positive constant independent of $\tau$ and $h,$ but depends on the norms of $u$ with regularity given in \eqref{ucnA}, as well as on the parameters $|\lambda|, |\Omega|$ and  $T$.
\end{thm}

If we further assume that $u$ satisfies
\begin{equation}\label{ucnB}
u\in {C}^{3}([0,T];L^{\infty}(\Omega))\cap {C}^{2}([0,T];W^{2,\infty}(\Omega))\cap{C}^{0}([0,T];W^{4,\infty}(\Omega)\cap H^1_0(\Omega)).
\end{equation}
Then, for the BDF2 scheme, we have the following error estimate result.
\begin{thm}\label{enzhao3} Assume $h\lesssim h_{min}$, then under
the assumption \eqref{ucnB}, there exist $h_0>0,\,\tau_0>0$ sufficiently small, when $0<h\le h_0,\,0<\tau\le \tau_0$,
  we have the following error estimates for the BDF2 scheme \eqref{lognlsbdf2}:
\begin{subequations}\label{errbndlefd}
\begin{align}
& \|e^{n}\|\le  C_{u}(\tau^{2}|\ln(\tau)| + h^{2}), \quad  n=0,1,\cdots, N,\label{errbndsifd-1}\\
&  |e^{n}|_{1} \le  C_{u}\tau^{-\frac{1}{2}}  ({\tau}^{2}|\ln(\tau)| + {h}^{2}), \quad  n=0,1,\cdots, N.\label{errbndsifd-2}
\end{align}
\end{subequations}
\end{thm}

\subsection{Proof of Theorem \ref{enzhao0} on the BDF1 scheme \eqref{lognlsbdf1}} To prove Theorem \ref{enzhao0}, we  first introduce and analyze the local truncation error of the BDF1 scheme.
\begin{lemma}\label{TnorderIEFD}
Under the assumption \eqref{ucnA},
the local truncation error $\xi^{n}\in{X}_{h}^{0}$ of the BDF1 scheme \eqref{lognlsbdf1} is defined as
\begin{equation}\label{trunT}
\xi^{n}_{j,k}:=\ri\delta^{+}_{t}U_{j,k}^{n}+\delta_{\nabla}^{2}{U}_{j,k}^{n+1}-2\lambda f(U_{j,k}^{n}),
 \quad (j,k)\in\mathcal{T}_{h},\;\; n=0,1,2,\ldots,N-1.
\end{equation}
Then  for all $0\le n\le N-1,$ we get
\begin{equation}\label{tnest}
\|\xi^{n}\|\le C_u(\tau+h^2).
\end{equation}
\end{lemma}
(See Appendix \ref{AppendixA} for the  proof of Lemma \ref{TnorderIEFD} ).

Now we are ready to establish the error estimates of the BDF1 scheme given in Theorem \ref{enzhao0}.
\begin{proof}[Proof of Theorem \ref{enzhao0}]
Subtracting \eqref{lognlsbdf1} from \eqref{trunT}, we get the following `error' equation:
 \begin{subequations}\label{erroreqn}
\begin{align}
&\ri\delta^+_te^n_{j,k}+\delta_{\nabla}^2 e^{n+1}_{j,k}=\xi^{n}_{j,k}+\lambda\mathcal{N}_{j,k}^{n},
\quad {(j,k)}\in\mathcal{T}_{h},\;\; n=0,1,2,\dots,N-1,\label{erroreqnA}
\end{align}
\end{subequations}
where $\mathcal{N}^{n}\in X_{h}^{0}$ with
$\mathcal{N}_{j,k}^{n}:=2f(U_{j,k}^{n})-2f(u^n_{j,k}),\quad {(j,k)}\in\mathcal{T}_{h}^{0}$.

Now, we split our proof of  Theorem \ref{enzhao0} into the following two steps.

{\bf Step 1. (Estimate of $||e^{n}||$ for the BDF1 scheme)} In this step, we aim to prove the error estimate in the discrete $L^{2}$ norm given in \eqref{errbndiefd-1} and the bound of the numerical solution in the discrete $H^{1}$ norm by mathematical induction.

Obviously, $||{e}^{0}||=0$  and $||{u}^{0}||_{\infty}=||{U}^{0}||_{\infty}$ is bounded.  Assume that \eqref{errbndiefd-1} is valid and $||{u}^{n}||_{\infty}$ is bounded by $1+||{U}^{n}||_{\infty}$ for $0\le n\le m$ with $m<N$, and we next prove that they are still hold  for $n=m+1$.
Multiplying both sides of \eqref{erroreqnA} by $h_{1}h_{2}\bar{e}^{n+1}_{j,k}$ and summing together for $(j,k)\in \mathcal{T}_{h}$, then taking the imaginary part, we have
\begin{equation}\label{realpartA}
\begin{split}
\frac{ \|e^{n+1}\|^2 -\|e^{n}\|^2 }{2\tau} + \frac{\|{e}^{n+1}-e^{n} \|^2 }{2\tau} =\lambda {\rm{Im}}(\mathcal{N}^{n}, {e}^{n+1})_{h}
+ {\rm{Im}} (\xi^{n} , {e}^{n+1})_{h}, \;\; {0}\le{n}<N.
\end{split}
\end{equation}

In order to estimate the first term on the right-hand side  of \eqref{realpartA},
{we take $\alpha\in[\frac{1}{2},1)$} and apply Cauchy-Schwarz inequality and Lemmas \ref{imbedding}, \ref{N2L2}, \ref{LipSC} for $0<\epsilon\le  \delta_\alpha$ to obtain  the following results:
\begin{equation}\label{Nest}
\begin{split}
&| \lambda {\rm{Im}}(\mathcal{N}^{n}, {e}^{n+1})_{h}| \le |\lambda| \cdot |{\rm{Im}}( \mathcal{N}^{n},e^{n})_{h}| + |\lambda|\cdot  |{\rm{Im}}( \mathcal{N}^{n}, {e}^{n+1}-{e}^{n})_{h}|  \\
\le& 2|\lambda| \cdot  \|{e}^{n}\|^{2} +  \frac{|\lambda|^{2} \tau} 2 \| \mathcal{N}^{n} \|^{2} + \frac{\| {e}^{n+1} - {e}^{n} \|^{2} }{2\tau}  \\
\le&   { {2} (b-a)^{1-\alpha} }|\lambda|^2 {\mathcal H}^{2}_\alpha(\epsilon) \,\tau \, \|{e}^{n}\|^{2\alpha} + 2 |\lambda| \big( |\lambda|\Upsilon^2(\epsilon,{ {\Lambda}_\infty}) \,\tau +1 \big) \, \|{e}^{n}\|^{2}  + \frac{\| {e}^{n+1} - {e}^{n} \|^2 }{2\tau},
\end{split}
\end{equation}
where $ {\mathcal H}_\alpha( \epsilon)$ and $ \Upsilon(\epsilon,{{\Lambda}_\infty}) $ are defined in \eqref{holderbnd} and \eqref{Hupsilon} with
\begin{equation}\label{GmmA}
 {{\Lambda}_\infty}=\max_{0\le n\le m}\big\{\|u^{n}\|_\infty, \|U^{n}\|_{\infty}\big\}.
\end{equation}
For the second term on the right-hand side  of \eqref{realpartA},  deducing from Lemma \ref{TnorderIEFD}, we have
\begin{equation}\label{lastest2}
\begin{split}
&|{\rm{Im}} (\xi^{n}, e^{n+1})_{h}| \le \frac{1}{4}\|{e}^{n+1}\|^2 + \| \xi^{n} \|^{2}
 \le \frac{1}{4}\|{e}^{n+1}\|^2 + C_u (\tau^2+h^4).
\end{split}
\end{equation}
Substituting \eqref{Nest}-\eqref{lastest2} into \eqref{realpartA} gives
\begin{equation*}\label{lastest01}
	\begin{split}
		 \|{e}^{n+1}\|^2 - \|{e}^{n}\|^2
		& \le  \frac{\tau}{2}\| {e}^{n+1} \|^2 + 4|\Omega|^{1-\alpha} |\lambda|^2
		\mathcal{H}^2_\alpha(\epsilon) {\tau}^{2}  \| {e}^{n} \|^{2\alpha}
		\\&\quad + 4|\lambda| \tau  \big(|\lambda| \Upsilon^2(\epsilon,{{\Lambda}_\infty} ){\tau}  +1 \big) \| {e}^{n} \|^2
		+ {C_u}\,\tau ({\tau}^{2}+{h}^{4}).
	\end{split}
\end{equation*}
Summing the above inequality for $n=0,1,2,\ldots,m$ leads to
\begin{equation}\label{realpart2}
	\begin{split}
		\|e^{m+1}\|^2  &\le  \|e^{0}\|^2 +  \frac{ \tau}{2} \| e^{m+1} \|^2 +  4|\Omega|^{1-\alpha} |\lambda|^2
		\mathcal{H}^2_\alpha(\epsilon)  \tau^2  \sum_{n=0}^{m} \| e^n \|^{2\alpha}
		\\ &\quad + \tau  \Big( \frac{1}{2} + 4 |\lambda| + 4 |\lambda|^2 \Upsilon^{2}(\epsilon,{{\Lambda}_\infty}) \tau    \Big) \sum_{n=0}^{m}  \| e^{n} \|^2
		 + {C_u}{T}(\tau^2+ {{h^{4})} }.
	\end{split}
\end{equation}

Then for $0 <\tau < 1$ we derive from \eqref{realpart2} that there is
\begin{equation}\label{Upsilonem100}
\|e^{m+1}\|^2\le  c_1 + c_2\tau \sum_{n=0}^{m}\|e^n\|^{2\alpha}
+  c_3 \tau \sum_{n=0}^{m}  \|e^{n}\|^2,
\end{equation}
{where $\alpha\in [1/2,1)$} and
\begin{equation}\label{c123}
\begin{split}
c_1&=  C_u \, \widehat C_\tau\, T(\tau^2+ h^{4}), \quad c_2= 4 \widehat C_\tau\, |\Omega|^{1-\alpha} |\lambda|^2   {\mathcal H}^2_\alpha(\epsilon)\tau,\\
c_3&=\widehat C_\tau\, \Big( \frac 1 2+ 4 |\lambda| +4 |\lambda|^2 \Upsilon^{2}(\epsilon,{{\Lambda}_\infty}) \tau \Big),
\end{split}
\end{equation}
with $\widehat C_\tau=1/(1-\tau/2).$
By  the nonlinear Gr\"{o}nwall's inequality \eqref{disgronwalleqs20}, we have
\begin{equation}\label{disgronwalleqs200}
\begin{split}
\|e^{m+1}\|^2
& \leq c_1 \big(1-\alpha^{-1}+\alpha^{-1}{(1+\alpha c_1^{\alpha-1}c_2{\tau}+c_3{\tau})^{ m+1}}\big)\\
&\le c_1\big(1-\alpha^{-1} +\alpha^{-1} {\rm exp}\big({{ (m+1){\tau}} (\alpha c_1^{\alpha-1}c_2+c_3)}\big) \big),
\end{split}
\end{equation}
where the inequality $(1+z)^{ m+1}\leq e^{{(m+1)}z}$ for $z\ge 0$ was used.

For $0<\epsilon\le \delta_\alpha,$  we can bound the factor in the exponential as follows
\begin{equation}\label{Bronwalleqs00200}
\begin{split}
&G_m:={{(m+1)}{\tau} \big(\alpha { c_2 }/{ c_1^{ 1 - \alpha } }+c_3\big)}= { (m+1)}\tau  \widehat C_\tau \big( 1/2+ 4 |\lambda| +4 |\lambda|^2 \Upsilon^{2}(\epsilon,{\Lambda_\infty}) \tau \big)\\[6pt] &\quad +  \frac{ 4 { (m+1)} \alpha  \widehat C_\tau |\Omega|^{1-\alpha} |\lambda|^2   \tau^{2\alpha} {\mathcal H}^2_\alpha(\epsilon)}{({C_u \widehat C_\tau}\,T)^{1-\alpha} (1+ \tau^{-2} h^{4})^{1-\alpha}}
\\[6pt]
&\le  \widehat C_\tau T \big( 1/2+ 4 |\lambda| +4 |\lambda|^2 \Upsilon^{2}(\epsilon,{{\Lambda}_\infty}) \tau \big) +
4 \alpha C_u^{\alpha-1} (\widehat C_\tau T)^{\alpha}  |\Omega|^{1-\alpha}  |\lambda|^2   \tau^{2\alpha-1} {\mathcal H}^2_\alpha(\epsilon),
\end{split}
\end{equation}
where we refer the inequality: $(1+z)^{\beta}\le 1$ for $z\ge 0$ and $\beta=\alpha-1<0.$

We now choose suitable values of $\epsilon$ in the upper bound of $G_m$ in \eqref{Bronwalleqs00200},  so as to obtain the best possible order of convergence.
Taking $\epsilon=\tau$ in ${\mathcal H}^2_\alpha(\epsilon)$ and $\Upsilon^{2}(\epsilon,{{\Lambda}_\infty})$, noticing from the definition \eqref{Hupsilon} and \eqref{GmmA}  that 
 for $0<\tau  \le \delta_\alpha,$ we get
 \begin{equation}\label{UptauA00}
 \begin{split}
& \tau^{2\alpha-1} {\mathcal H}^2_\alpha(\tau)\le 4^{1-\alpha} \tau (|\!\ln \tau|+1)^2,
 \;\;\; \tau \Upsilon^2(\tau, {{\Lambda}_\infty})\le 2\tau \big\{ | \! \ln  \tau|^2+ (\ln({\Lambda_\infty+1}) +1)^2 \big\}.
\end{split}
 \end{equation}
For $0<\tau  \le \delta_\alpha<1$, there is
   $$1\le \widehat C_\tau\le \widehat C_{\delta_\alpha}\le \frac 1{1-1/(2e)}< \frac {5} 4.$$
Therefore, for the above results, we can rewrite \eqref{Bronwalleqs00200} as
 \begin{equation}\label{GmC}
 \begin{split}
 G_m  \le \frac {5} 4 T\Big(\frac 1 2+ 4 |\lambda| \Big) + CT \big(1+(T C_u)^{\alpha-1}\big) \tau |\! \ln\tau |^2
  +C (C_u^{\alpha-1} T^{\alpha} + {  (\ln  ({\Lambda}_\infty+1) )^2 }) \tau,
  \end{split}
 \end{equation}
where $C_u$ is nonlocal and depend on the norms of $u$ with regularity given in \eqref{ucnA}.
Then we deduce from \eqref{c123}-\eqref{GmC} that, if $\frac 1 2\le \alpha<1,\, 0<\tau \le  \delta_\alpha$, there is
 \begin{equation}\label{Gmbnd2}
 \begin{split}
\|e^{m+1}\|\le &\sqrt{c_1(1 +\alpha^{-1}({\rm exp}(G_m)-1))}\\
\le & C_u   (\tau + {h^2 } )\,  {\rm exp}\big(CT\{1+(C_{u}T)^{\alpha-1}\tau | \! \ln \tau|^2 \} + (C_u + C{ (\ln  ({\Lambda}_\infty+1) )^2 } ) \tau \big).
\end{split}
\end{equation}
Noting that $\delta_\alpha={\rm exp}(\alpha/(\alpha-1))$ is decreasing in $\alpha$ and $\delta_\alpha\to 0$ as $\alpha\to 1^{-}$,   we then choose $\alpha=1/2$ so that $\delta_\alpha$ has the maximum value $e^{-1}.$ Moreover, as the function $\tau|\! \ln \tau |^2 $  is monotonically  increasing on $[0, \delta_\alpha] $, this implies
$$
0\le \tau|\! \ln \tau |^2\le \delta_\alpha |\! \ln \delta_\alpha |^2\le \delta_{1/2} |\! \ln \delta_{1/2}|^2=e^{-1}.
 $$
Besides, according to the assumption of induction, we obtain
  $$
  { \ln  (\Lambda_\infty+1) \le C_u + \ln (1 + \|U^{n} \|_{\infty}). }
  $$
Thus, we can conclude from \eqref{Gmbnd2} that
 \begin{equation}\label{Gmbnd20}
\|e^{m+1}\| \le  C_{u}  (\tau + h^2).
\end{equation}
Rewriting \eqref{erroreqnA} as follows
\begin{equation}\label{erroreqnAre}
\delta^2_\nabla e^{m+1}_{j,k}=-\ri\delta^+_te^m_{j,k}+\xi^{m}_{j,k}+\lambda\mathcal{N}_{j,k}^{m},
\end{equation}
and taking the discrete $L^2$ norm on both sides of \eqref{erroreqnAre},
we obtain
\begin{equation}\label{eH2}
\begin{split}
|e^{m+1}|_2&\le \tau^{-1}(\|e^m\|+\|e^{m+1}\|)+\|\xi^{m}\|+\lambda\|\mathcal{N}^{m}\|\\
&\le \tau^{-1}(\|e^m\|+\|e^{m+1}\|) +{ {2} |\Omega|^{(1-\alpha)/2} }|\lambda| {\mathcal H}_\alpha(\epsilon) \, \|e^m\|^{\alpha} \\
&\quad + 2 |\lambda|\Upsilon(\epsilon,{ {\Lambda}_\infty}) \, \|e^m\|+ C_u (\tau+h^2),
\end{split}
\end{equation}
then taking $\epsilon=\tau$ in  ${\mathcal H}_\alpha(\epsilon)$ and $\Upsilon(\epsilon,{ {\Lambda}_\infty})$, we find from \eqref{UptauA00}  that 
 for $0<\tau  \le \delta_\alpha$, there is
 \begin{equation}\label{UptauA00-1}
 \begin{split}
& {\mathcal H}_\alpha(\tau)\le (2\tau)^{1-\alpha} (|\!\ln \tau|+1),\;\;
  \Upsilon(\tau, {{\Lambda}_\infty})\le \sqrt{2} \big\{ | \! \ln  \tau|+ (\ln({\Lambda_\infty+1}) +1) \big\}.
\end{split}
 \end{equation}
Then \eqref{eH2} can be bounded as
\begin{equation}\label{eH2s}
\begin{split}
|e^{m+1}|_{2}&\le \tau^{-1}(\|e^m\|+\|e^{m+1}\|) +{ {2}^{2-\alpha}\tau^{2-\alpha} (|\!\ln \tau|+1) |\Omega|^{1-\alpha} }|\lambda|\, \|e^m\|^{\alpha} \\
&\quad + 2\sqrt{2} |\lambda|\big\{ | \! \ln  \tau|+ (\ln({\Lambda_\infty+1}) +1) \big\} \, \|e^m\|+ C_u (\tau+h^2)\\
&\le C_u \tau^{-1}(\tau+h^2).
\end{split}
\end{equation}
By Lemma \ref{uinflemma} and \eqref{eH2s}, we obtain
 \begin{equation}\label{einf-1}
 \begin{split}
\|e^{m+1}\|_{\infty} &\le C\|e^{m+1}\|^{\frac{1}{2}}(|e^{m+1}|_2+\|e^{m+1}\|)^{\frac{1}{2}}\\
&\le C_u (\tau^{-1}+1)^{\frac{1}{2}}(\tau+h^2)\le C_u(\tau^{\frac{1}{2}}+\tau^{-\frac{1}{2}}h^2).
\end{split}
\end{equation}
On the other hand, noticing $|e^{m+1}|_{2}\le{4}{h_{1}^{-1}h_{2}^{-1}}||e^{m+1}||$ and using Lemma \ref{uinflemma}, we have
 \begin{equation}\label{einf-2}
 \begin{split}
\|e^{m+1}\|_{\infty} &\le C\|e^{m+1}\|^{\frac{1}{2}}(|e^{m+1}|_2+\|u^{m+1}\|)^{\frac{1}{2}}\\
&\le C_u (h_{1}^{-\frac{1}{2}}h_{2}^{-\frac{1}{2}}+1)(\tau+h^2)\le C_u (h^{-1}\tau+h),
\end{split}
\end{equation}where $h\lesssim{h}_{min}$ was used.
Combining \eqref{einf-1} and \eqref{einf-2} gives that, without any constraint on the grid ratio, when $\tau$ and $h$ are sufficiently small, we always have $\|e^{m+1}\|_{\infty} \le {1}.$
This together with triangle inequality gives that
\begin{equation}\label{umbnde00}
\begin{split}
\|u^{m+1}\|_\infty  & \le  \|U^{m+1} \|_{\infty} + \|e^{m+1}\|_{\infty} \le 1+ \|U^{m+1} \|_{\infty}.
\end{split}
\end{equation}

{\bf Step 2. (Estimate of $|e^{n}|_{1}$ for the BDF1 scheme)} In this section, we try to establish the error estimate in the discrete $H^{1}$ norm. To do this, we fix $n=m$ in \eqref{erroreqnA},  and compute the inner product of the error equation  \eqref{erroreqnA} with  ${e}^{n+1}$,  then take the real part  to obtain
\begin{align}\label{realpartA-H1}
&|{e}^{n+1}|_{1}^2=\frac{1}{\tau}{\rm{Im}}(e^{n+1},e^{n})_{h}-\lambda {\rm{Re}}(\mathcal{N}^{n}, {e}^{n+1})_{h}
- {\rm{Re}} (\xi^{n} , {e}^{n+1})_{h}, \quad {0}\le{n}<N.
\end{align}
By using the Cauchy-Schwarz inequality and \eqref{ful2est}, \eqref{errbndiefd-1}, \eqref{tnest} \eqref{UptauA00-1},and \eqref{umbnde00}, we obtain
\begin{align}\label{realpartA-H1-2}
|{e}^{n+1}|_{1}^{2}\le&\frac{1}{2\tau}(||e^{n+1}||^{2}+||e^{n}||^{2})+\frac{|\lambda|}{2} (||\mathcal{N}^{n}||^{2}+||{e}^{n+1}||^{2})
+\frac{1}{2}(||\xi^{n}||^{2}+||{e}^{n+1}||^{2})\nonumber\\
\le&C_{u}(\tau^{-1}+1)(\tau+h^{2})^{2}+C_{u} (|\ln \tau|+1)^{2}(\tau+h^{2})^{2}\nonumber\\
\le&{C}_{u}\tau^{-1}(\tau+h^{2})^{2}, \quad  {n}=0,1,2,\ldots,N-1,
\end{align}
which immediately gives \eqref{errbndiefd-2}. And then the proof of Theorem \ref{enzhao0} is completed.
\end{proof}



\subsection{Proof of Theorem \ref{enzhao3} on the BDF2 scheme \eqref{lognlsbdf2}}\label{lognlsbdf2-sec}
The estimate of the local truncation error of the BDF2 scheme will be needed below, when we prove Theorem  \ref{enzhao3}.
\begin{lemma}\label{Tnorder-bdf2}
Under the assumption of the exact solution given in \eqref{ucnB},
the local truncation error $\widetilde{\xi}^{n}\in{X}_{h}^{0}$ of the BDF2 scheme \eqref{lognlsbdf2} is defined as
 \begin{subequations}\label{trunT-bdf2}
  \begin{align}
&\widetilde{\xi}^{n}_{j,k}:=\ri {D}_{t}^{-}{U}^{n+1}_{j,k}+\delta^2_{\nabla} {U}^{n+1}_{j,k}-2\lambda f(2U^{n}_{j,k}-U^{n-1}_{j,k}),
 \quad (j,k)\in\mathcal{T}_{h},\;\; {1}\leq{n}<{N},\label{lte-bdf2-1}\\
 &\widetilde{\xi}^{0}_{j,k}=\ri\delta^{-}_tU_{j,k}^{1}+\delta_{\nabla}^2 U_{j,k}^{1}-2\lambda f(U_{j,k}^{0}),\quad (j,k)\in\mathcal{T}_{h},\label{lte-bdf2-2}
\end{align}
\end{subequations}
then  we have
\begin{equation}\label{lte-bdf2-3}
\|\widetilde{\xi}^{0}\|\le C_u(\tau+h^2),\quad
\|\widetilde{\xi}^{n}\|\le C_u(\tau^{2}|\ln(\tau)|+h^2),\quad n=1,2,\ldots,N.
\end{equation}
\end{lemma}
(See Appendix B for the proof of Lemma \ref{Tnorder-bdf2})


Now we  give a rigorous proof of the error estimate results stated in Theorem  \ref{enzhao3}.
\begin{proof} [Proof of Theorem \ref{enzhao3}]
When $n=0$, we have $e_{j,k}^{0}=0$ for $(j,k)\in\mathcal{T}_{h}^{0}$. In order to establish the  estimate of $e^{1}$, we subtract  \eqref{lognlsLEFD-2} from  \eqref{lte-bdf2-2} to get the following error equation:
\begin{equation}\label{e1}
\frac{\ri}{\tau}e_{j,k}^{1}+\delta_{\nabla}^2 e_{j,k}^{1}=\widetilde{\xi}^{0}_{j,k}, \;\; (j,k)\in\mathcal{T}_{h}.
\end{equation}
Computing the inner product of \eqref{e1} with $e^{1}$ and taking the imaginary part, we obtain
\begin{equation*}
\frac{1}{\tau}\|e^{1}\|^2={\rm{Im}}(\widetilde{\xi}^{0},e^{1})_{h}\le \frac{\tau}{2}\|\widetilde{\xi}^{0}\|^2+\frac{1}{2\tau}\|e^{1}\|^2,
\end{equation*}
which implies
\begin{equation}\label{bdf2-err-bf-0}
\|e^{1}\|\le C_{u}{\tau} ({\tau}+{h}^{2})\le  C_{u}({\tau}^{2}|\ln(\tau)|+{h}^{2}).
\end{equation}
Computing the inner product of \eqref{e1} with $-e^{1}$ and taking the real part, we get
\begin{equation*}
|e^{1}|_1^2=-{\rm{Re}}(\widetilde{\xi}^{0},e^{1})_{h}\leq||\widetilde{\xi}^{0}||\cdot||e^{1}||\le C_{u}\tau({\tau}+{h}^{2})^{2}\le C_{u}\tau^{-1}({\tau}^{2}+{h}^{2})^{2},
\end{equation*}
which gives
\begin{align}\label{est-lev-1-1}
|e^{1}|_{1}\le C_{u}\tau^{-\frac{1}{2}}(\tau^2+h^2)
\end{align}
On the other hand, by using the inverse inequality $|e^{1}|_{1}\le {C}h^{-1}||e^{1}||$, we have
\begin{align}\label{est-lev-1-2}
|e^{1}|_{1} \le C_{u}{h}^{-1}(\tau^2+h^2).
\end{align}
From \eqref{est-lev-1-1} and \eqref{est-lev-1-2}, we know that for sufficiently small $h$ and $\tau$ we always have
\begin{align}\label{est-lev-1-3}
|e^{1}|_{1} \le {1}.
\end{align}
Therefore, the error estimate results given in Theorem \ref{enzhao3} are true for $n=0, 1$.

Next we use the mathematical induction to establish the error estimate in the discrete $L^{2}$ norm for $n>{1}$ and prove the boundedness of the error function in the discrete $H^{1}$ norm. By
subtracting \eqref{lognlsLEFD-1} from \eqref{lte-bdf2-1}, we have the following error equation:
\begin{equation}\label{erroreqnAsifd}
\ri{D}_{t}^{-}e^{n+1}_{j,k}+\delta_{\nabla}^2 e^{n+1}_{j,k}=\widetilde{\xi}^{n}_{j,k}+\lambda\widetilde{\mathcal{N}}_{j,k}^{n},
\quad {(j,k)}\in\mathcal{T}_{h},\;\; n=1,2,\dots,N,
\end{equation}
where $\widetilde{\mathcal{N}}^{n}\in X_{h}^{0}$ with
\begin{equation}\label{eheh0sifd}
\widetilde{ \mathcal{N}}_{j,k}^{n}:=2f(2U_{j,k}^{n}-U_{j,k}^{n-1})-2f(2u_{j,k}^{n}-u_{j,k}^{n-1}),\quad {(j,k)}\in\mathcal{T}_{h}^{0},\;\; n=1,2,\dots,N.\nonumber
\end{equation}
We assume that
 \begin{align}\label{bound-h1-sifd-pf}
 & \|e^{n}\|\le  C_{u}({\tau}^{2}|\ln(\tau)| + {h}^{2}),\quad   {0}\le{n}\le{m}
 \end{align}
 and
  \begin{align}\label{bound-h1-sifd}
 &|e^{n}|_{1}\le{C}_{u}\tau^{-\frac{1}{2}}(\tau^2|\ln(\tau)|+h^2),\quad
 |e^{n}|_{1} \le {1},\quad   {0}<n\le{m}
 \end{align}
 hold for $m\le{N}-1$ and next we will show that \eqref{bound-h1-sifd-pf} and \eqref{bound-h1-sifd} still valid for $n=m+1$.

Computing the inner product of \eqref{erroreqnAsifd} with ${4\tau}{e}^{n+1}$ and taking the imaginary part, we have
 \begin{equation}\label{imagepartAsifd}
 \begin{split}
&  \| e^{n+1} \|^{2} + \|2e^{n+1}-e^{n}\|^{2} - \| e^{n} \|^{2} - \|2e^{n}-e^{n-1}\|^{2} + \| e^{n+1}+2e^{n+1}-e^{n}\|^{2}  \\
=\;& {4\lambda\tau}{\rm{Im}}(\mathcal{\widetilde{N}}^{n},2{e}^{n}-{e}^{n-1})_{h}
  + {4\lambda\tau}{\rm{Im}}(\mathcal{\widetilde{N}}^{n},{e}^{n+1}-2{e}^{n}+{e}^{n-1})_{h}
 + {4\tau}{\rm{Im}}(\widetilde{\xi}^{n},{e}^{n+1})_{h}.
 \end{split}
 \end{equation}
For the first term on the right hand side of  \ref{imagepartAsifd}, by using Lemma \eqref{LipSC}, we have the following estimate:
\begin{align}\label{bdf2-err-bf-1}
&{4\lambda\tau}{\rm{Im}}(\mathcal{\widetilde{N}}^{n},2{e}^{n}-{e}^{n-1})_{h}
\leq {4}{\tau}|\lambda| \|2{e}^{n}-{e}^{n-1}\|^{2}.
\end{align}
For the second term on the right hand side of  \ref{imagepartAsifd}, by using Cauchy-Schwarz inequality, we have the following estimate:
\begin{align}\label{bdf2-err-bf-2}
&{4\lambda\tau}{\rm{Im}}(\mathcal{\widetilde{N}}^{n},{e}^{n+1}-2{e}^{n}+{e}^{n-1})_{h}
\leq  |\lambda|{\tau}^{2}\|\widetilde{N}^{n}\|^{2} +\|{e}^{n+1}-2{e}^{n}+{e}^{n-1}\|^{2}.
\end{align}
By using Lemma \ref{LipSC-1eq-2} and taking $\varepsilon=\tau^{2}$, we have
\begin{align}\label{bdf2-err-bf-3}
& \|\widetilde{N}^{n}\| \leq C_{u}(\tau^{2} + |\ln(\tau)| \|2{e}^{n}-{e}^{n-1}\| ),
\end{align}
this together with \ref{bdf2-err-bf-2} gives
\begin{align}\label{bdf2-err-bf-4}
&{4\lambda\tau}{\rm{Im}}(\mathcal{\widetilde{N}}^{n},{e}^{n+1}-2{e}^{n}+{e}^{n-1})_{h}\nonumber\\
\leq\;  & C_{u}{\tau}(\tau^{4} + {\tau}|\ln(\tau)|^{2} \|2{e}^{n}-{e}^{n-1}\|^{2} )
  +\|{e}^{n+1}-2{e}^{n}+{e}^{n-1}\|^{2}\nonumber\\
\leq\;  & C_{u}{\tau}(\tau^{4} +  \|2{e}^{n}-{e}^{n-1}\|^{2} )
  +\|{e}^{n+1}-2{e}^{n}+{e}^{n-1}\|^{2},
\end{align}
where ${\tau}|\ln(\tau)|^{2}\leq{4e^{-2}}$ $(\tau<1)$ was used.
For the third term on the right hand side of  \ref{imagepartAsifd}, by using Cauchy-Schwarz inequality, we have the following estimate:
\begin{align}\label{bdf2-err-bf-5}
&{4\tau}{\rm{Im}}(\widetilde{\xi}^{n},{e}^{n+1})_{h} \leq  {4\tau}\|\widetilde{\xi}^{n}\|^{2} + {\tau}\|{e}^{n+1}\|^{2}.
\end{align}
Substituting \eqref{bdf2-err-bf-1}, \eqref{bdf2-err-bf-3}, \eqref{bdf2-err-bf-4} into \eqref{imagepartAsifd} gives
 \begin{equation} \label{bdf2-err-bf-5}
 \begin{split}
&  \| e^{n+1} \|^{2} + \|2e^{n+1}-e^{n}\|^{2} - \| e^{n} \|^{2} - \|2e^{n}-e^{n-1}\|^{2}   \\
\leq\;  & C_{u}{\tau}(\tau^{4} + \|{e}^{n}\|^{2} + \|2{e}^{n}-{e}^{n-1}\|^{2}) + {\tau}\|{e}^{n+1}\|^{2} + {2\tau} \|\widetilde{\xi}^{n}\|^{2}.
 \end{split}
 \end{equation}
  Summing \eqref{bdf2-err-bf-5} for $n=1,2,\ldots,m$, we have
 \begin{equation}\label{bdf2-err-bf-6}
 \begin{split}
 &\| e^{m+1} \|^{2} + \|2e^{m+1}-e^{m}\|^{2}\nonumber\\
 \le\; & {5}\| e^{1} \|^2+\tau \| e^{m+1} \|^2+{C}_{u}\tau \sum_{n=1}^{m}\left(\tau^{4} + \|{e}^{n}\|^{2} + \|2{e}^{n}-{e}^{n-1}\|^{2}\right)
  +{4}\tau \sum_{n=1}^{m}\|\widetilde{\xi}^{n}\|^{2}.
 \end{split}
 \end{equation}
 Hence, if we set $\tau\leq{1/2}$ and use \eqref{lte-bdf2-3} and \eqref{bdf2-err-bf-0}, we have
  \begin{equation}\label{bdf2-err-bf-6}
 \begin{split}
 &\| e^{m+1} \|^{2} + \|2e^{m+1}-e^{m}\|^{2}\nonumber\\
 \le\; &  {C}_{u}\tau \sum_{n=1}^{m}\left(\|{e}^{n}\|^{2} + \|2{e}^{n}-{e}^{n-1}\|^{2}\right)
  +C_{u}(\tau^{2}|\ln(\tau)|+h^{2})^{2}.
 \end{split}
 \end{equation}
  Then we use Gr\"{o}nwall inequality to obtain that
\begin{equation}\label{bdf2-err-bf-8}
\| e^{m+1} \| \le {C_{u}}({\tau}^{2}|\ln(\tau)|+h^2).
\end{equation}

Next, we try to estimate $|e^{m+1}|_1$. To do this, we set $n=m$ in \eqref{erroreqnAsifd} and compute the inner product of \eqref{erroreqnAsifd} with $-e^{m+1}$, then take the real part of the result to obtain
 \begin{equation*}\label{realpartAsifd}
 \begin{split}
 &{\rm{Im}}(D_{t}^{-}e^{m+1},{e}^{m+1})  +  | e^{m+1} |_{1}^{2}
 = \lambda {\rm{Re}}(\mathcal{\widetilde{N}}^{m},{e}^{m+1})+ {\rm{Re}}(\widetilde{\xi}^{n},{e}^{m+1}),
 \end{split}
 \end{equation*}
 which implies
  \begin{equation}\label{bdf2-err-bf-9}
  \begin{split}
  | e^{m+1} |_{1}^{2}  &\le \frac{3}{\tau} (\|{e}^{m+1}\|^2+\|{e}^{m}\|^2+\|{e}^{m-1}\|^2)
   + {C_{u}}({\tau}^{2}|\ln(\tau)|+h^2)^{2} \\
   &\leq {C_{u}}{\tau}^{-1}({\tau}^{2}|\ln(\tau)|+h^2)^{2},
   \end{split}
    \end{equation}
where  \eqref{lte-bdf2-3}, \eqref{bdf2-err-bf-3}, \eqref{bound-h1-sifd-pf} and \eqref{bdf2-err-bf-8}  were used. It follows from \eqref{bdf2-err-bf-9} that
  \begin{equation}\label{bdf2-err-bf-10}
  | e^{m+1} |_{1}
 \leq {C_{u}}{\tau}^{-\frac{1}{2}}({\tau}^{2}|\ln(\tau)|+h^2).
    \end{equation}
The above inequality implies that when  $h\le \tau$ there holds
\begin{equation*}
| e^{m+1}|_{1}\le C_{u} \tau^{-\frac{1}{2}}|\ln\tau|(\tau^2+h^2)\le C_{u} (\tau^{\frac{3}{2}}|\ln(\tau)|+h^{\frac{3}{2}}).
\end{equation*}
On the other hand, when $\tau\le h$, we obtain from the inverse inequality $| e^{m+1}|_{1}\le{2}h^{-1}\| e^{m+1}\|$ that
\begin{equation*}
| e^{m+1}|_{1} \le C_{u} {h}^{-1}(\tau^{2}|\ln(\tau)|+h^{2})\le C_{u}(\tau|\ln(\tau)|+h).
\end{equation*}
Thus, for sufficiently small $h$ and $\tau$,  we always have
\begin{equation*}
| e^{m+1}|_{1}\le {1}.
\end{equation*}
This ends the proof.
\end{proof}
\begin{remark}
The proofs of Theorems  {\rm \ref{enzhao0}}- {\rm \ref{enzhao3}} can be generalized to one- or
three-dimensional cases by utilizing the proposed numerical analysis techniques.
\end{remark}
\begin{remark}
We discretize the LogSE equation \eqref{lognls0B-2d} in space with the compact difference method and we can get the convergence order $O(h^4)$ in one, two or
three dimensions.
\end{remark}
\begin{remark}
The BDF1 \eqref{lognlsbdf1}, BDF2 \eqref{lognlsbdf2} schemes for  the LogSE \eqref{lognlss}  can be extended to solve the initial-boundary value problem of the LogSE with a potential,
\begin{equation}\label{lognlss}
\ri  \partial_t u(\boldsymbol{x},t)+ \Delta u(\boldsymbol{x},t) +  V(\boldsymbol{x})u(\boldsymbol{x},t)=\lambda u (\boldsymbol{x},t)\ln(|u(\boldsymbol{x},t)|^2), \quad   \boldsymbol{x}\in  \Omega, \;\; t>0,\\[2pt]
\end{equation}
where $V\in L^{\infty}(\Omega)$ is a known  real-valued function, and the
same convergence results with Theorems  {\rm \ref{enzhao0}}- {\rm \ref{enzhao3}} can be obtained.
\end{remark}
\begin{remark}
In addition to the provided proof for Theorem \eqref{enzhao0}, we can use a similar proof  procedure for Theorem \eqref{enzhao3} to prove Theorem \eqref{enzhao0}.
\end{remark}

\section{Numerical results}
In this section, we numerically confirm the theoretical convergence results given in Theorems {\rm \ref{enzhao0}}, \ref{enzhao3} and simulate the dynamics of solutions. Here, we fix $\lambda =-1,\, d = 2.$ For more efficient computation, we apply FFT Sine wave transformation to solve the difference equations at homogeneous Dirichlet boundary condition. In more detail, we define $f=(f_1,f_2,\ldots,f_{J-1})^{T}$, the centered
difference operator can be rewritten as
\begin{equation}
(D_x^2 f)_j=(Af)_j,\quad A=
\begin{bmatrix}
-\frac{2}{h^2}& \frac{1}{h^2} & &\\
\frac{1}{h^2} &-\frac{2}{h^2} & \frac{1}{h^2} & \\
& & \ldots &\frac{1}{h^2}\\
& &\frac{1}{h^2} & -\frac{2}{h^2}
\end{bmatrix}_{J-1,J-1},
\end{equation}
where $x\in (-1,1).$

In fact, $A$ has the following complete set of orthogonal eigenvectors:
\begin{equation}
\boldsymbol{v}_k=\begin{bmatrix}
\sin(k\pi x_1)\\
\sin(k\pi x_2)\\
\ldots\\
\sin(k\pi x_{J-1})
\end{bmatrix},
\quad 1\le k\le J-1.
\end{equation}
And $A$ has eigenvalues
\begin{equation}
\lambda_k=\frac{-4\sin\frac{k\pi h}{2}}{h^2},\quad 1\le k\le J-1.
\end{equation}
And then in the computational, we can apply FFT and inverse FFT.
\begin{remark}
The finite difference method has relatively high requirements on the regularity of solutions, so it has certain limitations when dealing with such low-regularity problems. In future work, we will focus on studying computational algorithms suitable for low-regularity initial values.
\end{remark}
\subsection{Accuracy test on exact Gausson solutions}
We first consider the LogSE \eqref{lognls0B} with the exact Gaussion solution \cite{Carles2018Universal}:
\begin{equation}\label{2dGex}
u(\boldsymbol{x},t)=\exp(2\ri \lambda \omega t+\omega+1+\lambda/2|\boldsymbol{x}|), \; \boldsymbol{x}\in \mathbb{R}^{2},\,t\geq 0,
\end{equation}
for any period $\omega \in \mathbb{R}$.

  Firstly, we test the convergence order of different FDTD schemes in two dimensions on the domain $\Omega=[-5,5]^2$. To demonstrate the convergence rate in time, we fix $h_x=h_y=0.001$ and take the time step as $\tau_j=0.1\times2^{-j},\,j=1,\ldots,4$ at $t=0.5$. Figure \ref{logsetimeorder} shows $\|e\|:=\max\limits_{1\le{n}\le{N}}\|e^{n}\|$ (left) and $|e|_{1}:=\max\limits_{1\le{n}\le{N}}|e^{n}|_{1}$ (right).

  Next, to examine the convergence order  in space, we fix $\tau=10^{-5}$ and vary the
spatial mesh size ${h_{x}}_j={h_{y}}_j=\frac{1}{4+2\times j},\,j=1,\ldots,5$ for the two schemes at $t=0.5$ with $\|e\|$ and $|e|_1$ (see Figure \ref{logsespaceorder}).
From Figures \ref{logsetimeorder} and \ref{logsespaceorder}, we conclude that (i) the BDF1 scheme converges to LogSE at ${\mathcal{O}(\tau+h^2)}$ in the discrete $L^2$-norm and $H^1$-norm which demonstrate Theorem \ref{enzhao0}; (ii) the BDF2 scheme converges quadratically ${\mathcal{O}(\tau^2+h^2)}$ in  the discrete $L^2$-norm and $H^1$-norm which confirm our error estimates in Theorem \ref{enzhao3}; (iii) From Figure \ref{logsetimeorder} (left), we can see the convergence orders of the BDF1 and BDF2 schemes are unconditional, optimal and almost optimal.

\begin{figure}[!h]
  \centering
   \subfigure{\includegraphics[height=1.7in,width=0.45\textwidth]{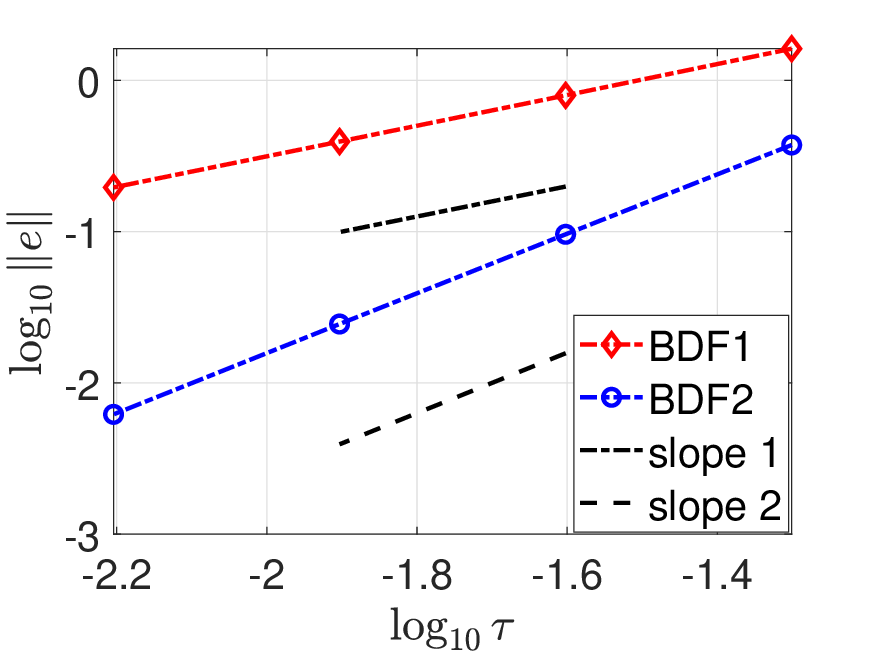}}
   \subfigure{\includegraphics[height=1.7in,width=0.45\textwidth]{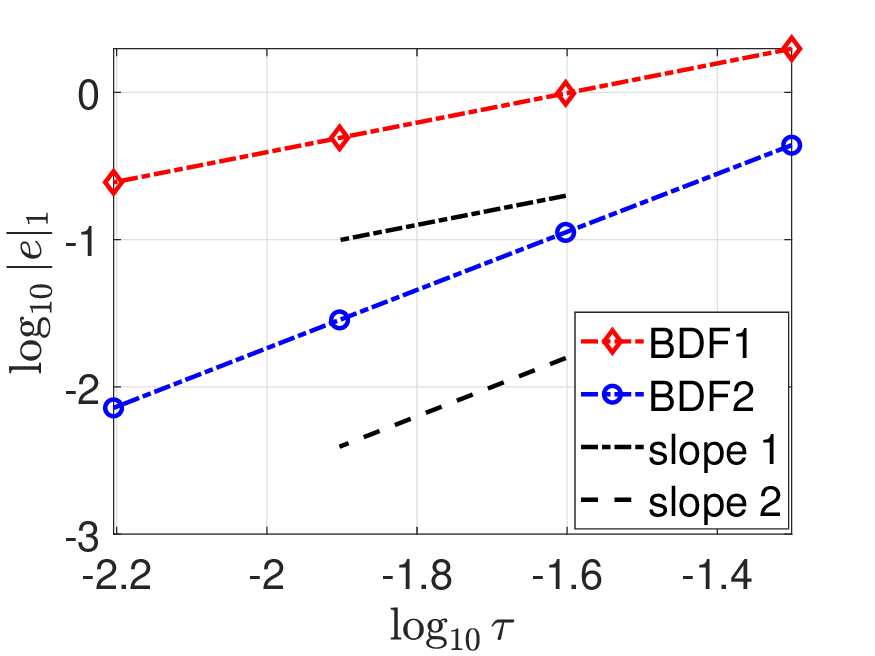}}
   \vspace{-0.05in}
    \caption{\small Convergence rates of the three schemes in time for the LogSE. }\label{logsetimeorder}
\end{figure}
\begin{figure}[!h]
  \centering
   \subfigure{\includegraphics[height=1.7in,width=0.45\textwidth]{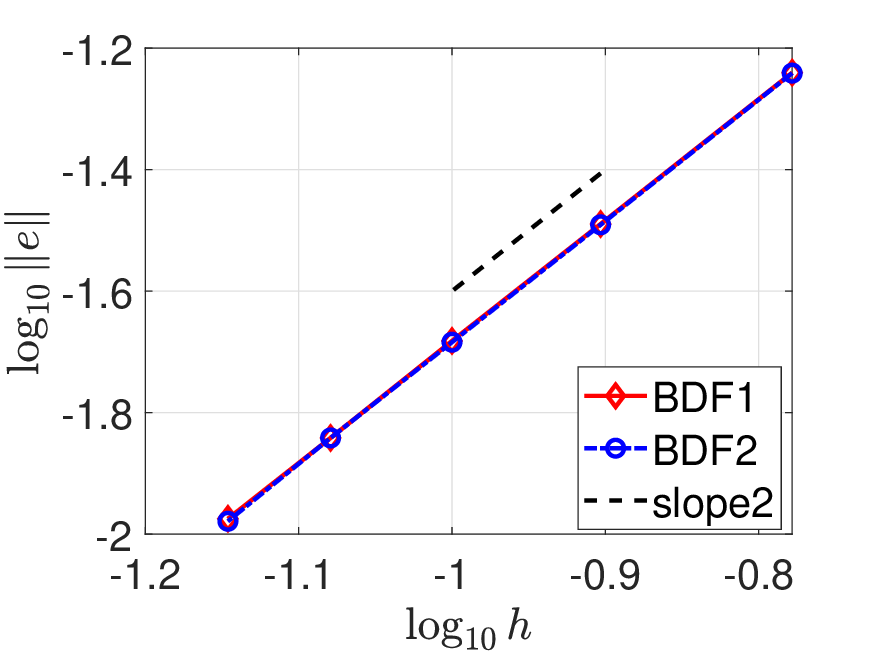}}
   \subfigure{\includegraphics[height=1.7in,width=0.45\textwidth]{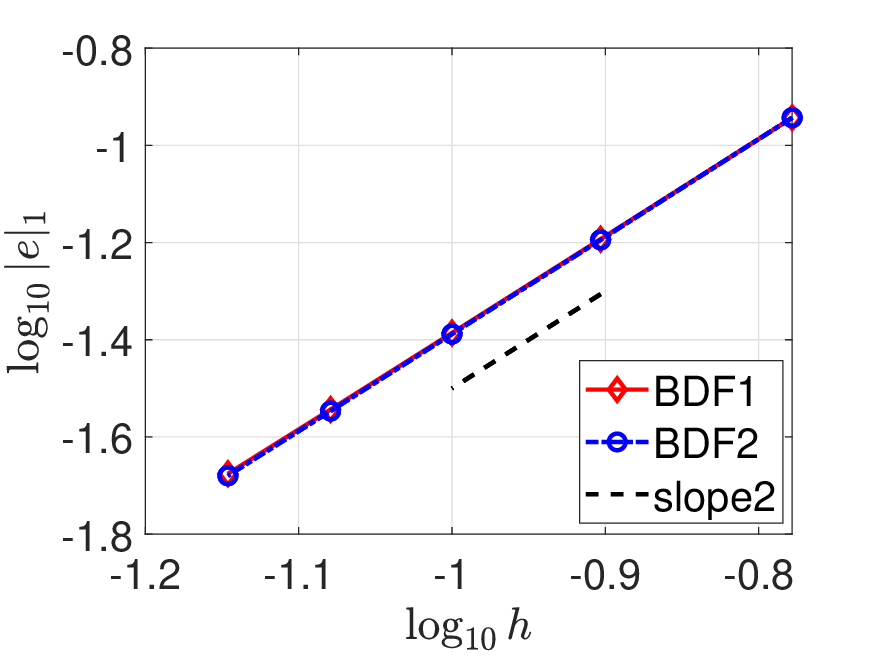}}
   \vspace{-0.05in}
    \caption{\small Convergence rates of the three schemes in space for the LogSE. }\label{logsespaceorder}
\end{figure}

\subsection{Application for the dynamics in 2D}
 In this section, we compare the dynamics of different initial solutions in the following examples.\\
 \textbf{Example 2.} Firstly, we present a single Gausson, a vortex pair and a vortex dipole in
 two dimensions (2D). Here, the initial data in the three cases are given as \cite{bao2022regularized}.\\
 Case I. A single Gausson, i.e.,
 \begin{equation*}
 u_0(x,y)=e^{-(x^2+y^2)};
 \end{equation*}
 Case II. A vortex pair, i.e.,
  \begin{equation*}
 u_0(x,y)=(x-0.5+\ri y)(x+0.5+\ri y)e^{-(x^2+y^2)};
 \end{equation*}
 Case III. A vortex pair, i.e.,
  \begin{equation*}
 u_0(x,y)=(x-0.5+\ri y)(x+0.5-\ri y)e^{-(x^2+y^2)}.
 \end{equation*}
We solve the problem by the BDF1 method \eqref{lognlsbdf1} with $\tau= 0.01, h_x =h_y =\frac{1}{32}$  and the computational domain $\Omega=[-8,8]^2$. The results of the BDF2 scheme is similar, and we omit them. We take $\lambda= -10$,
for Case I. $\lambda= 1$ for Case II and Case III.
 \begin{figure}[htbp]
  \centering
    \includegraphics[width=0.32\textwidth]{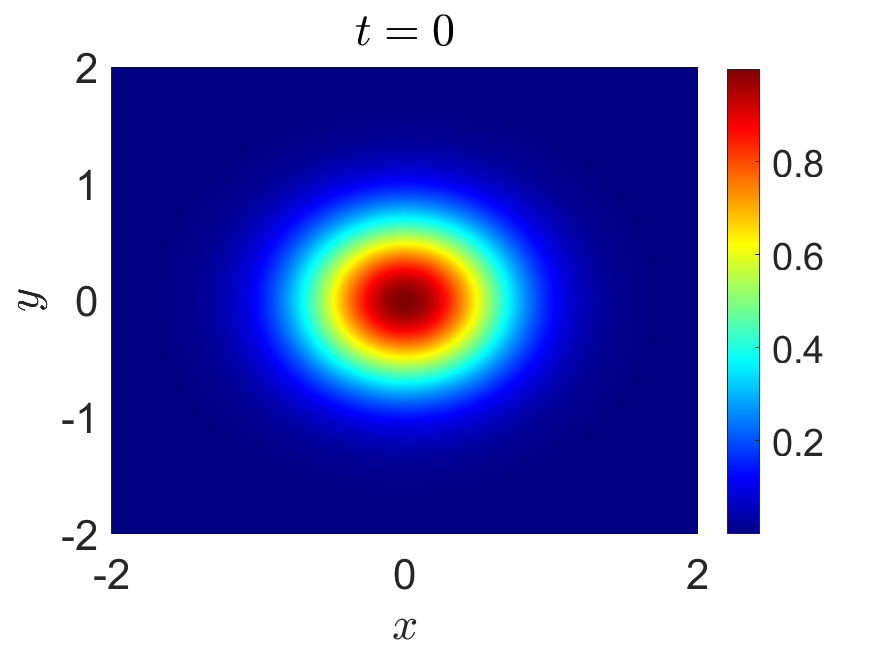}
    \includegraphics[width=0.32\textwidth]{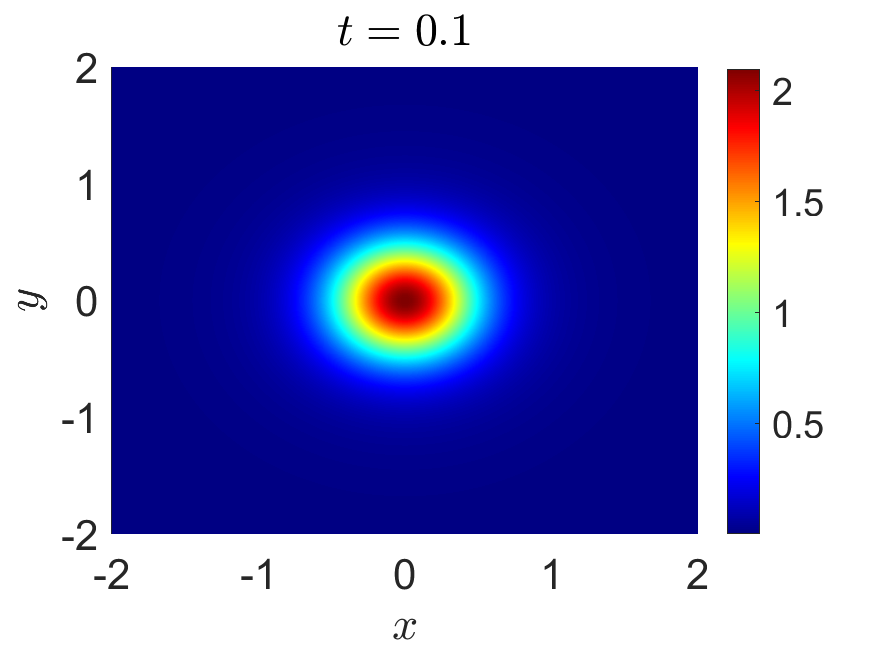}
    \includegraphics[width=0.32\textwidth]{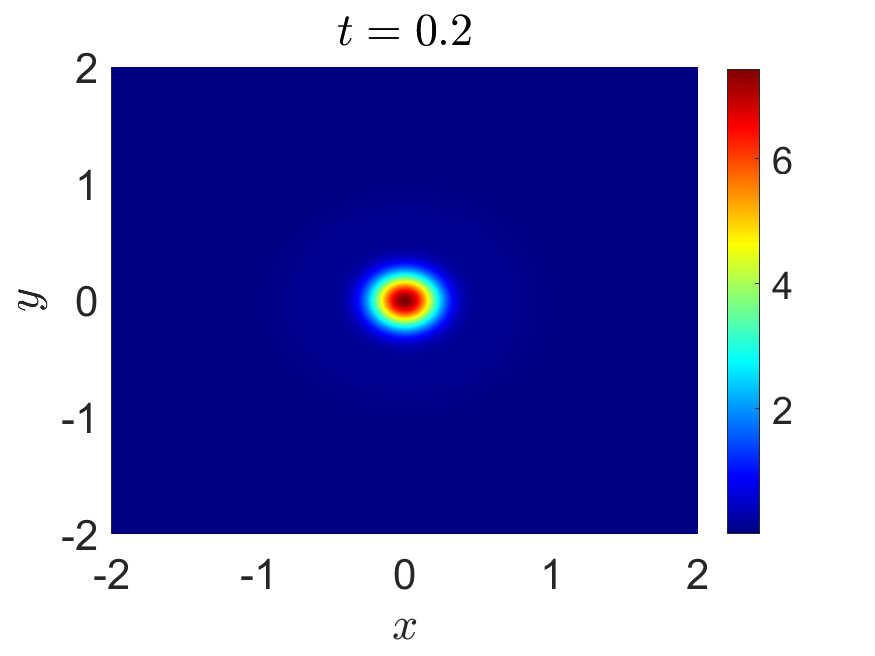}
    \includegraphics[width=0.32\textwidth]{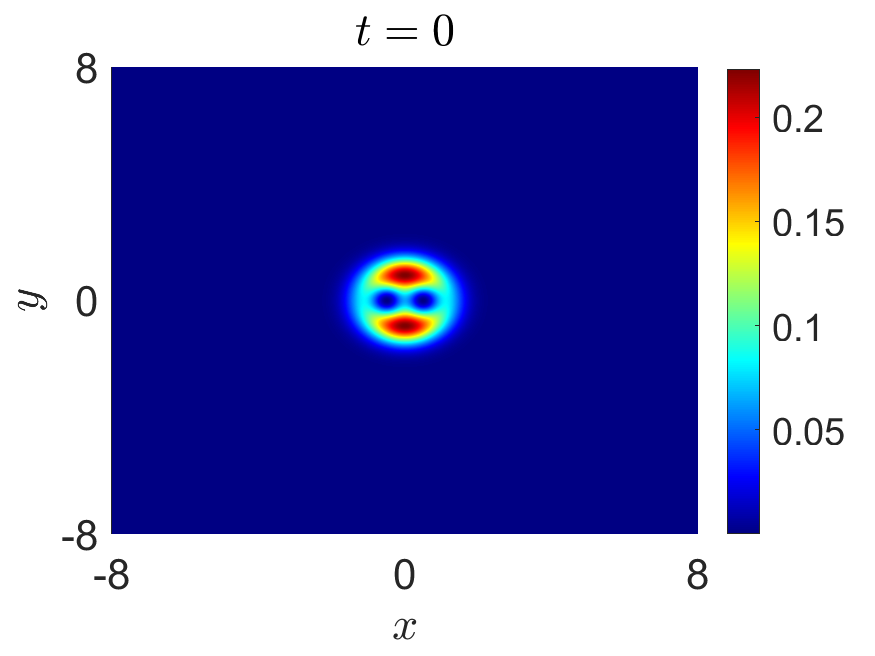}
    \includegraphics[width=0.32\textwidth]{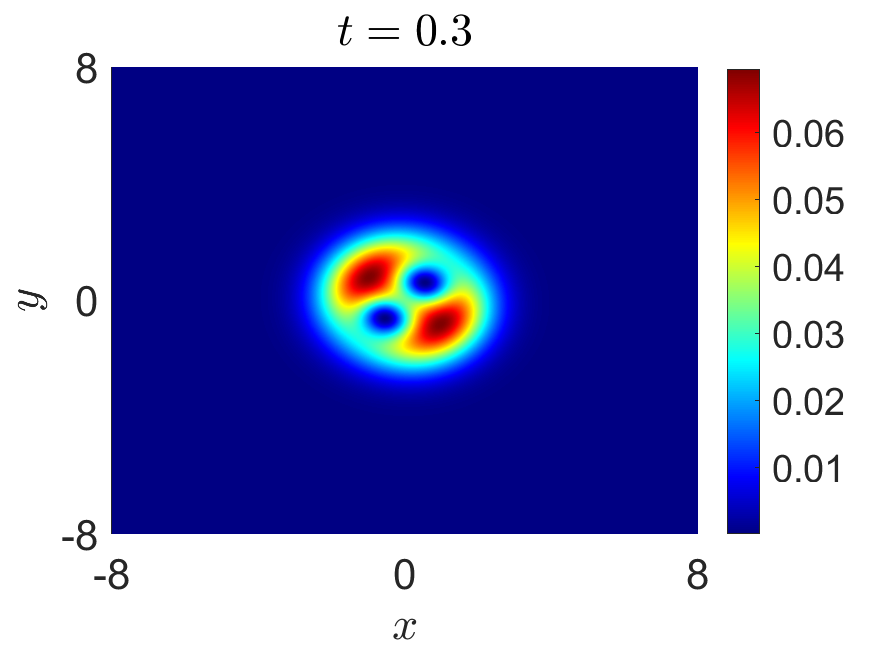}
    \includegraphics[width=0.32\textwidth]{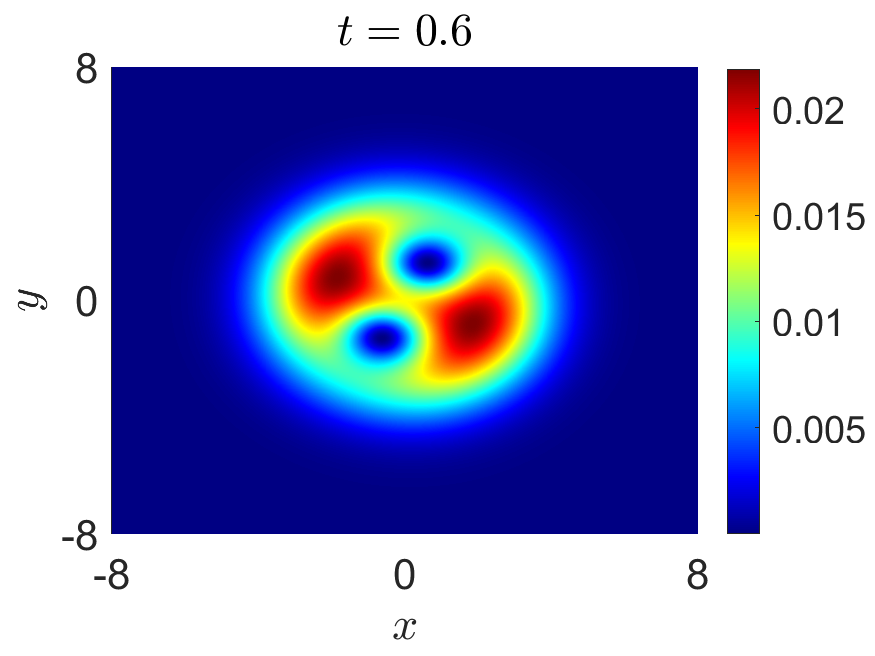}
    \includegraphics[width=0.32\textwidth]{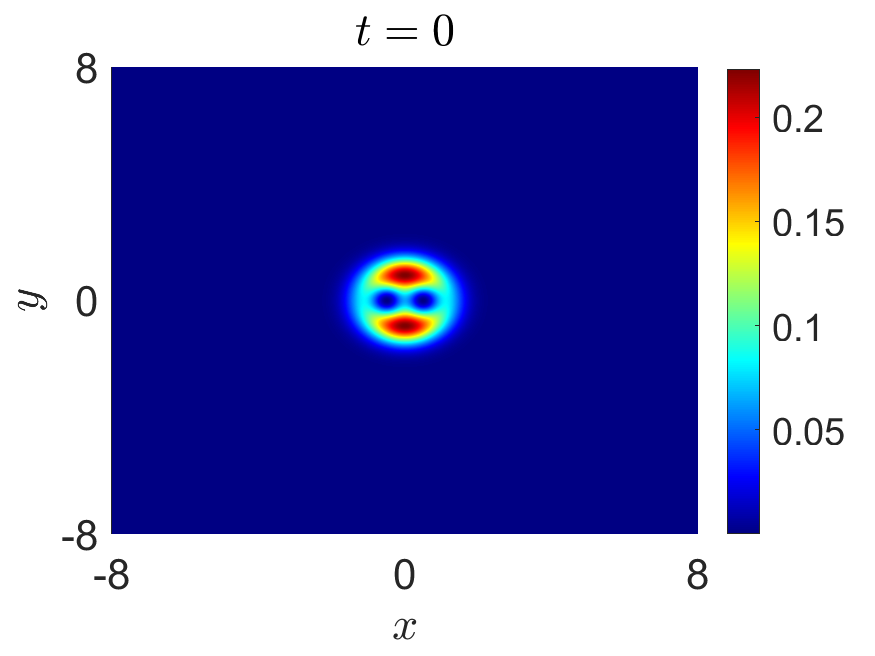}
    \includegraphics[width=0.32\textwidth]{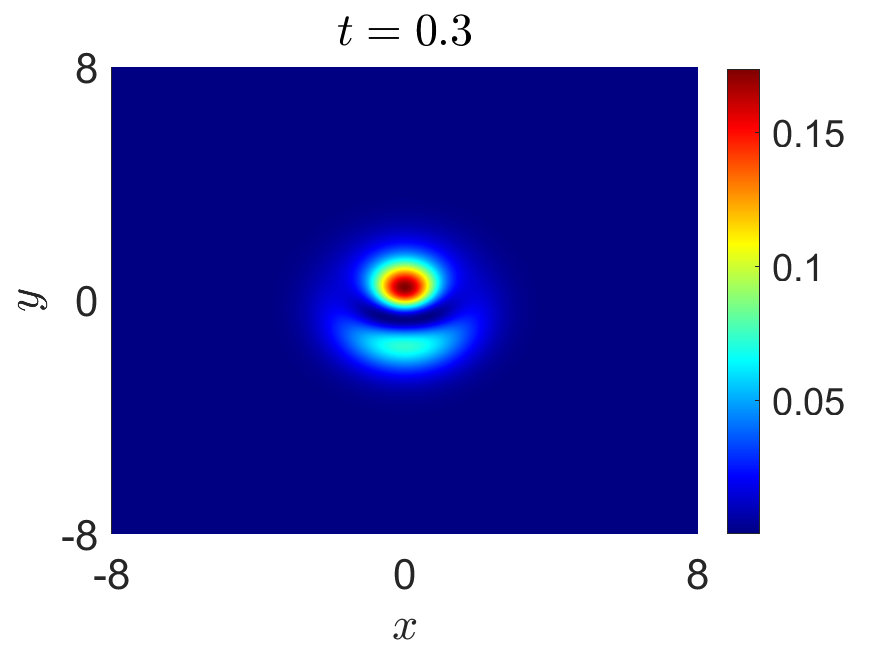}
    \includegraphics[width=0.32\textwidth]{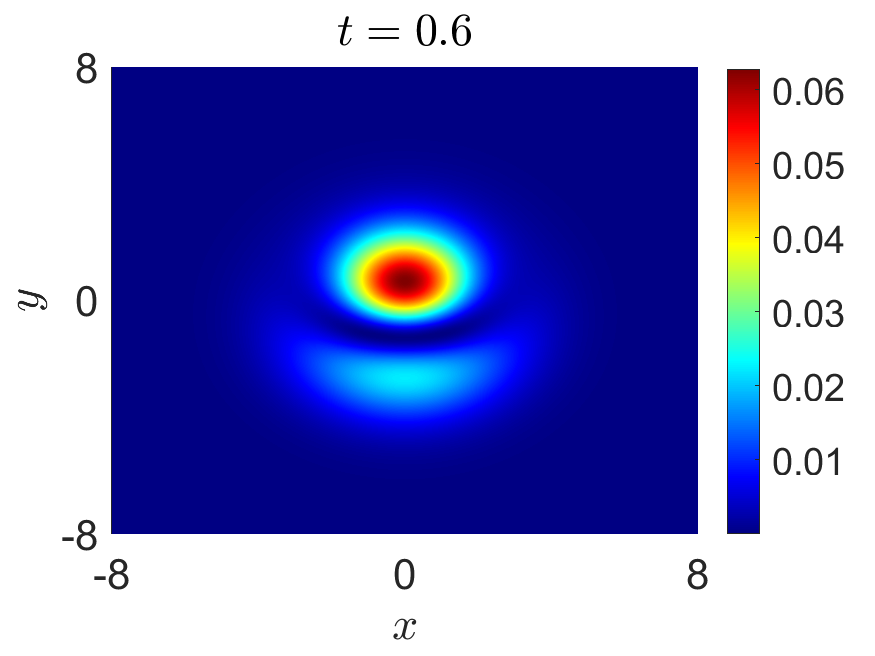}
    \caption{Plots of the density $\rho(x, y,t)$ at different times for Case I (top row); Case II (middle row); Case III (bottom row).}\label{u2dynamics}
\end{figure}
 From these figures, we can draw the following conclusions:
\begin{itemize}
 \item[(i)]  In Case I, we consider the focusing case of a single Gausson, i.e., $ \lambda< 0$. When
the time $t$ evolves, the single Gausson concentrates and the peak value becomes larger with the time evolution. (cf. top row of Figure \ref{u2dynamics})
\item[(ii)]  In Case II, we initially have a vortex pair located at $(\pm0.5,0)$.
 When the time $t$ evolves, the vortex pair rotates with each other and they never collide and annihilate
 (cf. middle row of Figure \ref{u2dynamics}).
  \item[(iii)] In Case III, we initially have a vortex dipole located at $(\pm0.5,0)$. When the time $t$ evolves, the two vortices start moving together, and then they collide and
 disappea within a short time (cf. bottom row of Figure \ref{u2dynamics}).
\end{itemize}
\textbf{Example 3.} Secondly, we simulate the interaction between the two Gaussons by using BDF2 scheme \eqref{lognlsbdf1} with $ \lambda=-1,\,\tau=0.001,\,h_x=h_y=\frac{1}{16},\,\Omega=[-16,16]^2$ for Cases IV and V while $\Omega=[-48,48]^2$ for Case VI. The initial solution \cite{Bao2022Error} is taken as follows
\begin{equation}\label{2dGi}
u_0(\boldsymbol{x})=b_1 e^{\ri \boldsymbol{x} \cdot \boldsymbol{v}_1+\frac{\lambda}{2}|\boldsymbol{x}-\boldsymbol{x}_1^0|^2}+b_2 e^{\ri \boldsymbol{x} \cdot \boldsymbol{v}_2+\frac{\lambda}{2}|\boldsymbol{x}-\boldsymbol{x}_2^0|^2},
\end{equation}
where $\boldsymbol{b}_j,\,\boldsymbol{v}_j,\,\boldsymbol{x}_j^0\,(j=1,2)$ are real constant vectors. Here, we consider the following cases:
\begin{itemize}
  \item [(IV)]$b_1=b_2=\frac{1}{\sqrt[4]{\pi}}, \boldsymbol{v}_1=\boldsymbol{v}_{\mathbf{2}}=(0,0)^T, \boldsymbol{x}_1^0=-\boldsymbol{x}_2^0=(-2,0)^T;$
  \item [(V)]$b_1=1.5 b_2=\frac{1}{\sqrt[4]{\pi}}, \boldsymbol{v}_1=(-0.15,0)^T, \boldsymbol{v}_2=\boldsymbol{x}_1^0=(0,0)^T, \boldsymbol{x}_2^0=(5,0)^T$;
  \item [(VI)]$b_1=b_2=\frac{1}{\sqrt[4]{\pi}}, \boldsymbol{v}_1=(0,0)^T, \boldsymbol{v}_{\mathbf{2}}=(0,0.85)^T, \boldsymbol{x}_1^0=-\boldsymbol{x}_2^0=(-2,0)^T$.
\end{itemize}
Figures \ref{u2IVdynamics}-\ref{u2VIdynamics} pictured the contour plots of the density $\rho(x, y,t)$ at different times for Cases IV, V, VI. From these pictures, we can observe that
if the two static Gaussons stay close enough, they will attract, collide and stick together shortly then separate again liking a pendulum (cf. Figure \ref{u2IVdynamics}). This phenomenon is similar to that in one dimension \cite{Bao2019Regularized}. In Figure \ref{u2Vdynamics}, we can observe that the moving Gausson will cause the static one to move in the same direction. However, if these two Gaussons staying close enough (cf. Figure \ref{u2VIdynamics}), if the moving Gausson moves perpendicular to the line connecting the two Gaussons, it will cause the static Gausson to be dragged into motion, and the direction of the moving Gausson will be altered as a result. Consequently, the two Gaussons will begin to rotate around each other, gradually drifting apart in the process.

\begin{figure}[htbp]
  \centering
    \includegraphics[width=0.32\textwidth]{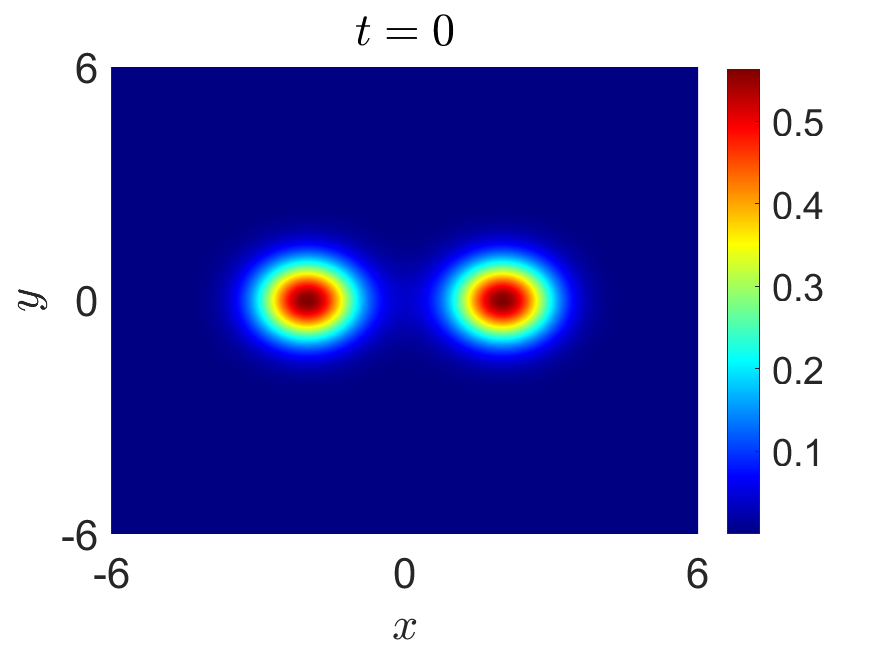}
    \includegraphics[width=0.32\textwidth]{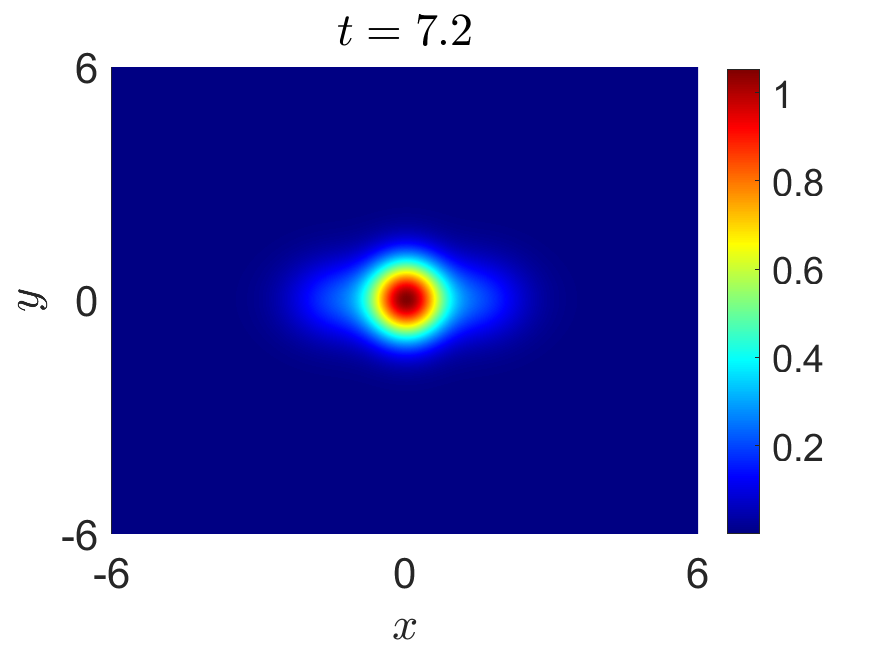}
    \includegraphics[width=0.32\textwidth]{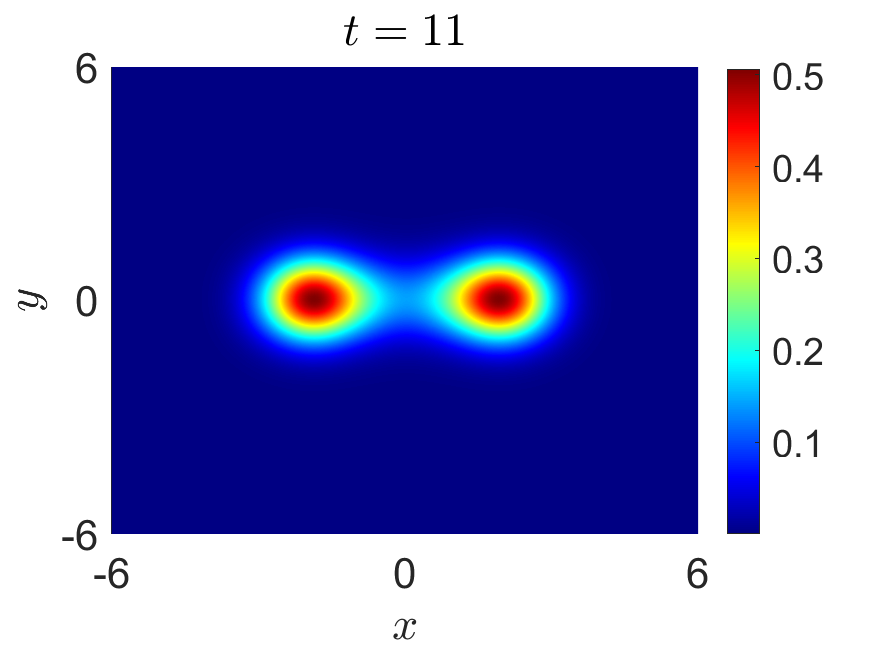}
    \caption{Plots of the density $\rho(x, y,t)$ at different times in the region $[-6,6]^2$ for Case IV.}\label{u2IVdynamics}
\end{figure}
\begin{figure}[htbp]
  \centering
    \includegraphics[width=0.32\textwidth]{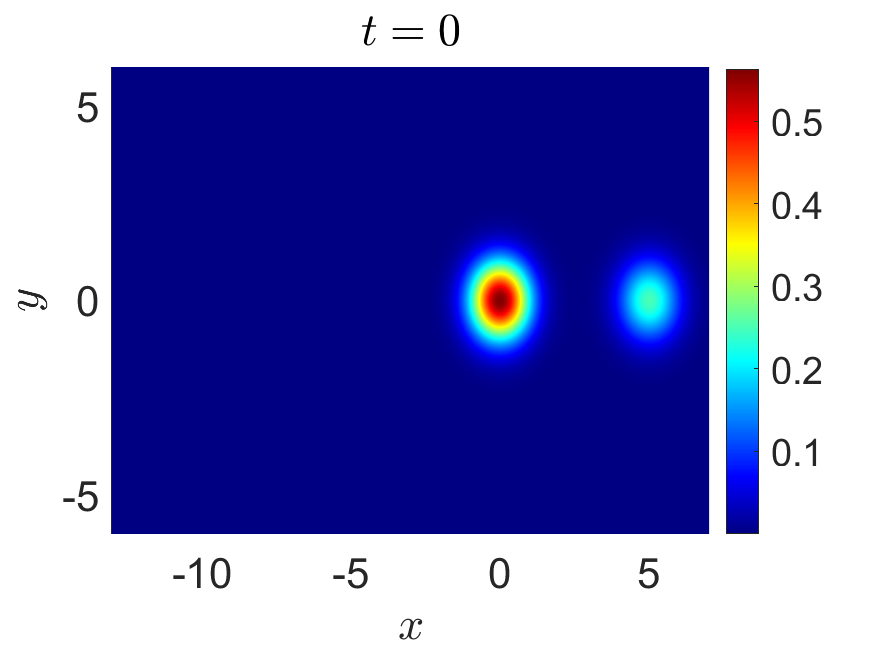}
    \includegraphics[width=0.32\textwidth]{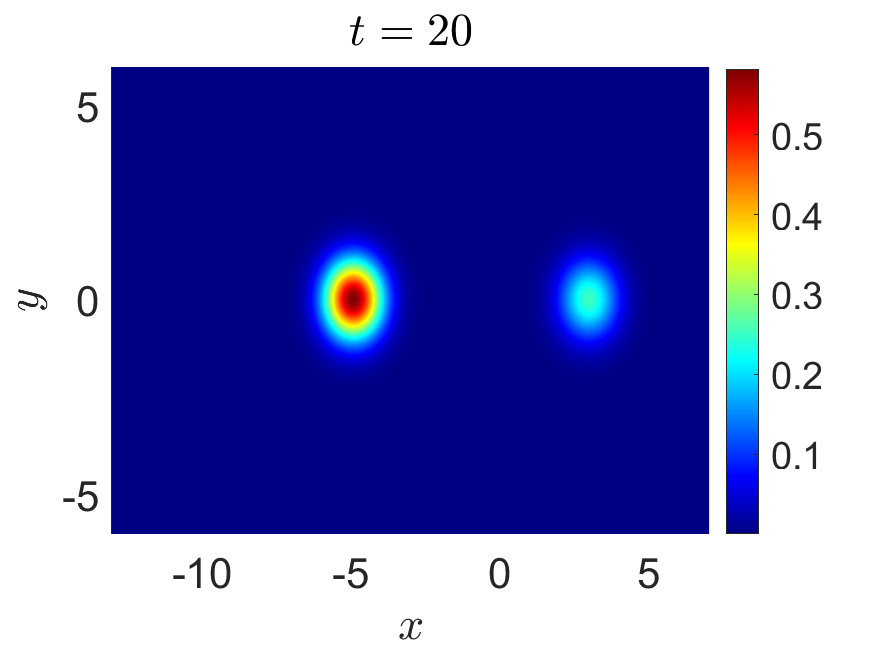}
    \includegraphics[width=0.32\textwidth]{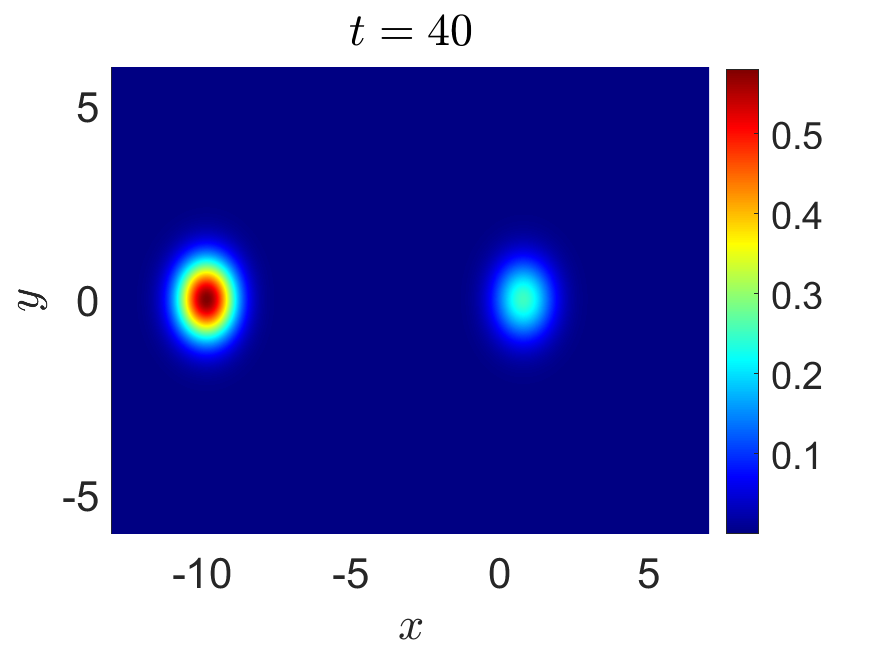}
    \caption{Plots of the density $\rho(x, y,t)$ at different times in the region $[-13,7]\times[-6,6]$ for Case V.}\label{u2Vdynamics}
\end{figure}
\begin{figure}[htbp]
  \centering
    \includegraphics[width=0.32\textwidth]{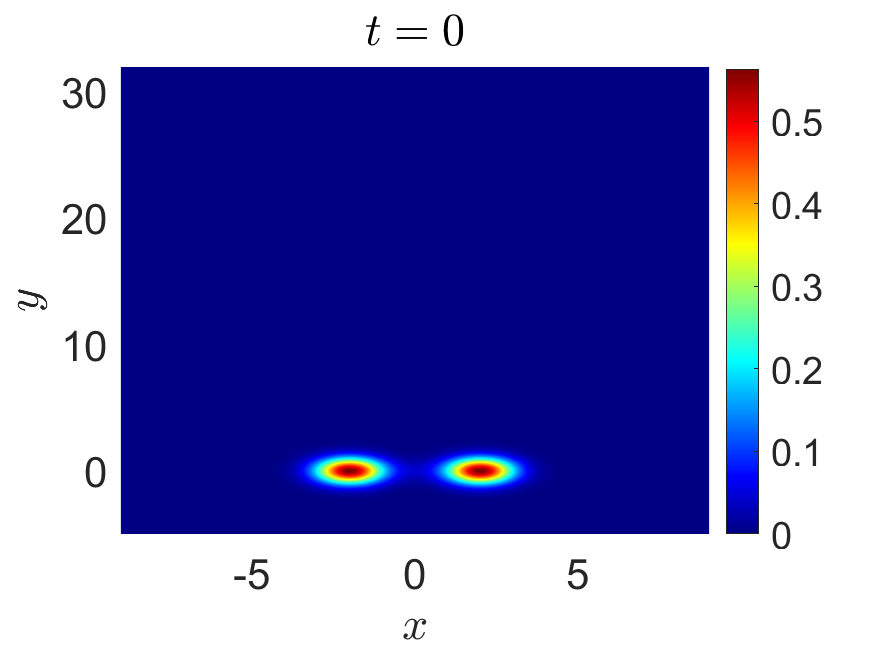}
    \includegraphics[width=0.32\textwidth]{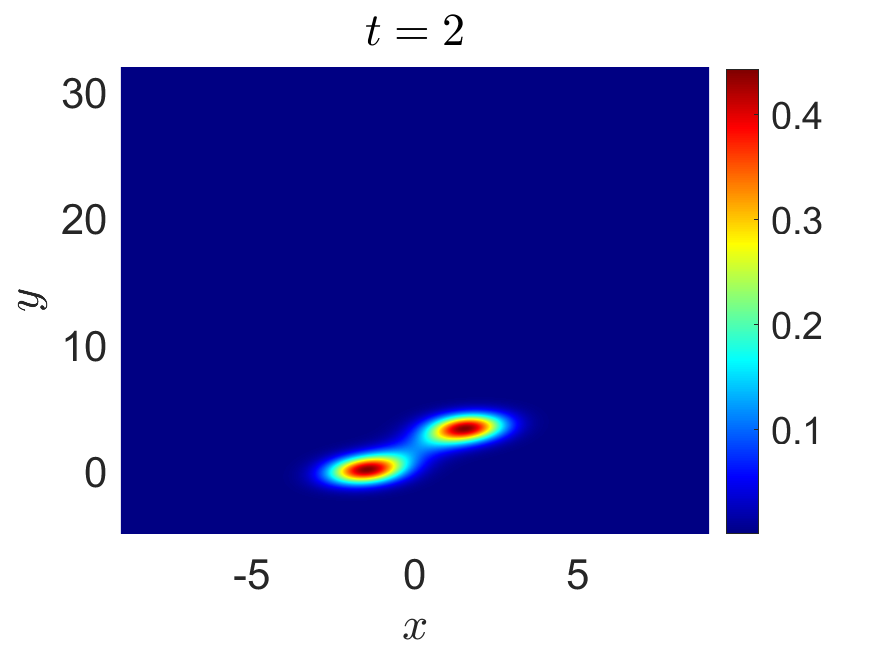}
    \includegraphics[width=0.32\textwidth]{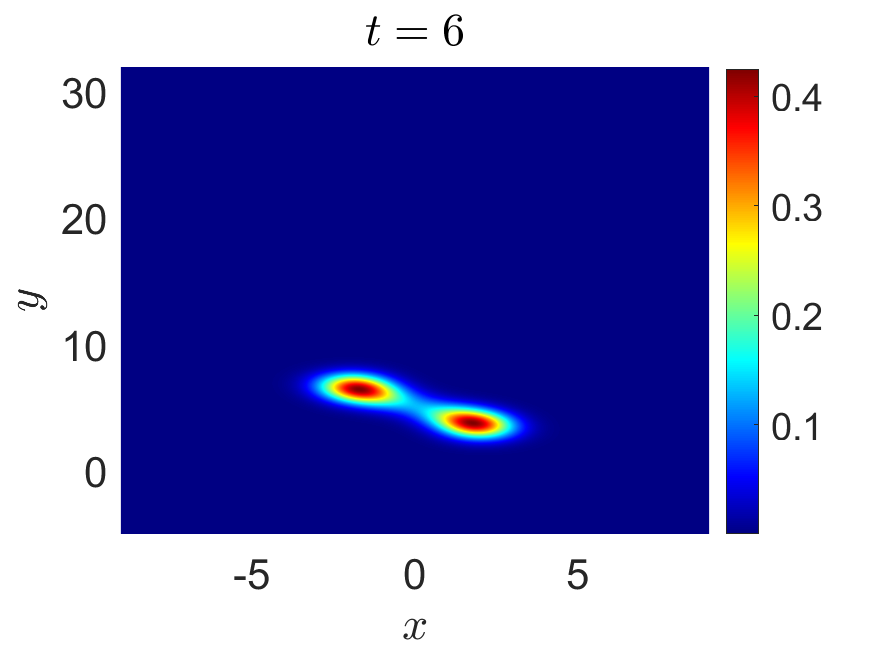}
    \includegraphics[width=0.32\textwidth]{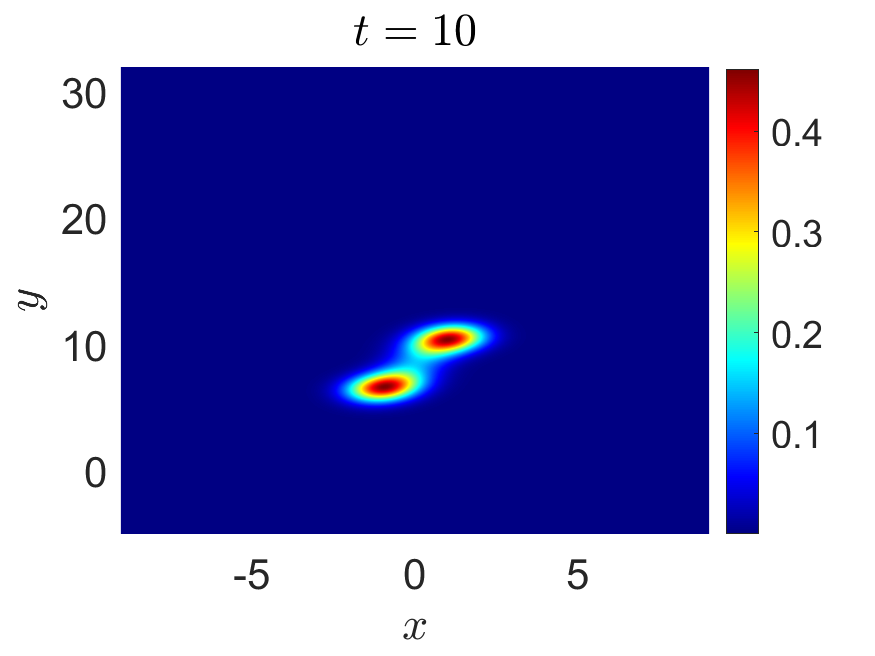}
    \includegraphics[width=0.32\textwidth]{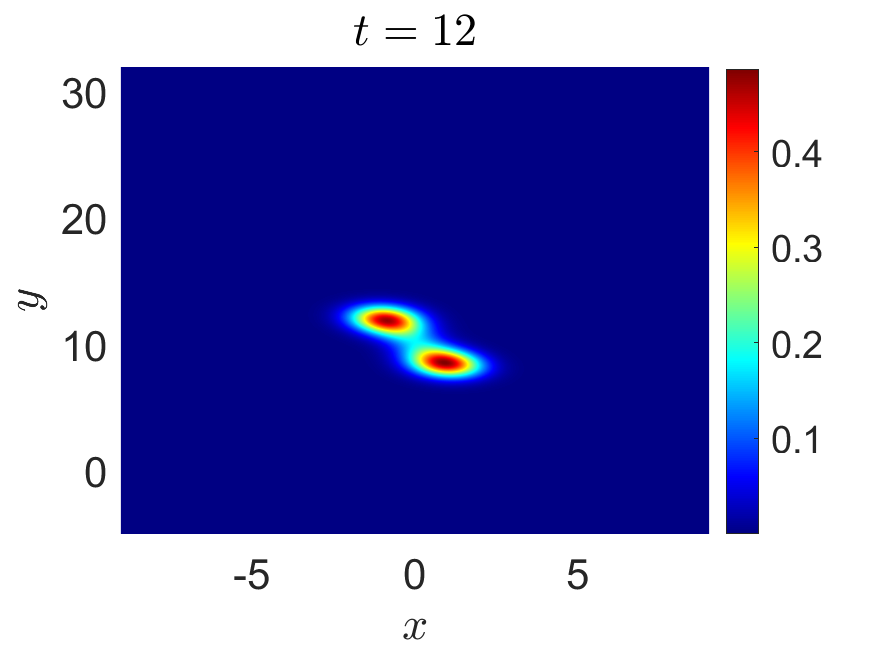}
    \includegraphics[width=0.32\textwidth]{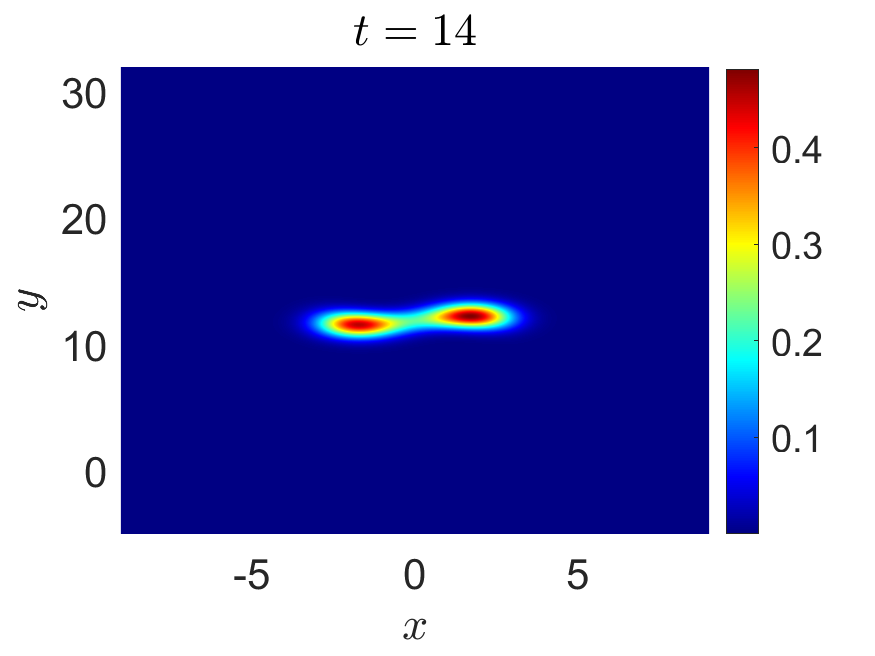}
    \caption{Plots of the density $\rho(x, y,t)$ at different times in the region $[-9,9]\times[-5,32]$ for Case VI.}\label{u2VIdynamics}
\end{figure}

\section{Concluding remarks}
In this paper, we propose the approximation of the solution to the LogSE by BDF1 and BDF2 schemes. We overcome the logarithmic nonlinear term bringing significant difficulties in the convergence analysis of the FDTD methods.
 The convergence of two schemes are established by proving optimal/ almost optimal  and unconditional order in the discrete $L^2$ norm with some introduced inequalities. Moreover, we use the lifting technique to avoid the constrain of mesh conditions. Establishing the optimal (not almost optimal) error estimate of the two schemes in the discrete $H^1$ norm will be the work we need to address in the near future.

\bigskip

\begin{appendix}

\setcounter{equation}{0}
	\renewcommand{\theequation}{A.\arabic{equation}}
	\section{Proof of Lemma \ref{TnorderIEFD}}\label{AppendixA}
Performing Taylor's expansion on the unknown function $u$ at the point $(x_{j},y_{k},t_{n})$ and applying some simple substitutions to the integral remainder, we have
\begin{equation*}
\begin{split}
\xi^{n}_{j,k}&= {\rm{i}}  \tau\alpha^n_{j,k} +  \tau  \beta^n_{j,k}+ \frac{h^2 }{6} \gamma^n_{j,k},
\end{split}
\end{equation*}
where
\begin{equation*}
\begin{split}
  \alpha^n_{j,k}&=\int^{1}_{0}(1-s)u_{tt}(x_j,y_k,t_n+s\tau)\,\mathrm{d} s,\\
  \beta^n_{j,k}&= \int^{1}_{0} \partial^2_{x}\partial_{t}  u(x_j,y_k,t_n+s\tau)\, \mathrm{d}s+\int^{1}_{0} \partial^2_{y}\partial_{t}  u(x_j,y_k,t_n+s\tau)\, \mathrm{d}s, \\
\gamma^n_{j,k}&= \int^{1}_{-1}(1-|s|)^3 \partial^4_{x} u(x_j+sh_x,y_k,t_{n+1})\, \mathrm{d}s+\int^{1}_{-1}(1-|s|)^3 \partial^4_{y} u(x_j,y_k+sh_y,t_{n+1})\, \mathrm{d}s
\end{split}
\end{equation*}
under the assumption of \eqref{ucnA}, then using the triangle inequality, we conclude that
\begin{equation*}
		\begin{split}
			|\xi^{n}_{j,k}|&\lesssim \tau \|\partial^{2}_{t} u\|_{L^{\infty}}+\tau \|\partial^{2}_{x}\partial_{t}  u\|_{L^{\infty}}+\tau \|\partial^{2}_{y}\partial_{t}  u\|_{L^{\infty}}+h_x^2 \|\partial^{4}_{x} u\|_{L^{\infty}}+h_y^2 \|\partial^{4}_{y}u\|_{L^{\infty}}\\
&\le \tau (\|\partial^{2}_{t} u\|_{L^{\infty}}+\|\partial^{2}_{x}\partial_{t}  u\|_{L^{\infty}}+\|\partial^{2}_{y}\partial_{t}  u\|_{L^{\infty}})+h^2 (\|\partial^{4}_{x} u\|_{L^{\infty}}+\|\partial^{4}_{y}u\|_{L^{\infty}})
\le C_u(\tau+h^2).
		\end{split}
	\end{equation*}
	This implies \eqref{tnest}.

\setcounter{equation}{0}
	\renewcommand{\theequation}{B.\arabic{equation}}
\section{Proof of Lemma \ref{Tnorder-bdf2}}\label{AppendixC}

Performing Taylor's expansion on each term of \eqref{lte-bdf2-1} $u$ at the point $(x_{j},y_{k},t_{n+1})$ and applying some simple substitutions to the integral remainder, we have
  \begin{align}
&\widetilde{\xi}^{n}_{j,k}:=\ri {D}_{t}^{-}{U}^{n+1}_{j,k}+\delta^2_{\nabla} {U}^{n+1}_{j,k}-2\lambda f(2U^{n}_{j,k}-U^{n-1}_{j,k}),
 \quad (j,k)\in\mathcal{T}_{h},\;\; {1}\leq{n}<{N},\label{pf-lte-bdf2-0}
\end{align}
where
\begin{align}
{D}_{t}^{-}{U}^{n+1}_{j,k}&=\partial_{t}u(x_{j},y_{k},t_{n+1})
-\tau^{2}\int_{0}^{1}[\partial_{t}^{3}u(x_{j},y_{k},t_{n+1}-s\tau)\nonumber\\
&-\partial_{t}^{3}u(x_{j},y_{k},t_{n+1}-2s\tau)](1-s)^{2}\;\rd{s},\label{pf-lte-bdf2-1}\\
\delta^2_{\nabla} {U}^{n+1}_{j,k}&=\Delta{u}(x_{j},y_{k},t_{n+1})
+\frac{h_{x}^2 }{6} \int^{1}_{-1}(1-|s|)^3 \partial^4_{x} u(x_j+sh_x,y_k,t_{n+1})\, \mathrm{d}s\nonumber\\
&+\frac{h_{y}^2 }{6} \int^{1}_{-1}(1-|s|)^3 \partial^4_{y} u(x_j,y_k+sh_y,t_{n+1})\;\rd{s},\label{pf-lte-bdf2-2}
\end{align}
and
\begin{align}
&|f(2U^{n}_{j,k}-U^{n-1}_{j,k})-f(U^{n+1}_{j,k})|\nonumber\\
\leq \;&|f(2U^{n}_{j,k}-U^{n-1}_{j,k})-f_{\varepsilon}(2U^{n}_{j,k}-U^{n-1}_{j,k})|\nonumber\\
   &+|f_{\varepsilon}(2U^{n}_{j,k}-U^{n-1}_{j,k})-f_{\varepsilon}(U^{n+1}_{j,k})|+|f_{\varepsilon}(U^{n+1}_{j,k})-f(U^{n+1}_{j,k})|\nonumber\\
=  \; &4{\varepsilon}+ 2|\ln(\varepsilon)| |U^{n+1}_{j,k}-2U^{n}_{j,k}+U^{n-1}_{j,k}|\nonumber\\
\leq \; &4{\varepsilon}+ 2|\ln(\varepsilon)| \tau^{2}\int_{0}^{1}|\partial_{t}^{2}u(x_{j},y_{k},t_{n}+s\tau)+\partial_{t}^{2}u(x_{j},y_{k},t_{n}-s\tau)|(1-s)\;\rd{s}\nonumber\\
\leq \; &C(\tau^{2}+ \tau^{2}|\ln(\tau)|),\label{pf-lte-bdf2-3}
\end{align}
under the assumption of \eqref{ucnB}, then using the triangle inequality and taking $\varepsilon=\tau^{2}$, we conclude that
\begin{equation*}
	\begin{split}
|\widetilde{\xi}^{n}_{j,k}|&\leq  C_{u}(\tau^{2} (1+\|\partial^{3}_{t}u\|_{L^{\infty}})+\tau^{2}|\ln(\tau)| \|\partial^{2}_{t}u\|_{L^{\infty}}+h_x^2 \|\partial^{4}_{x} u\|_{L^{\infty}}+h_y^2 \|\partial^{4}_{y}u\|_{L^{\infty}})\\
&\le C_{u}\tau^{2} (1+\|\partial^{3}_{t}u\|_{L^{\infty}}+|\ln(\tau)| \|\partial^{2}_{t}u\|_{L^{\infty}})+C_{u}{h}^{2} (\|\partial^{4}_{x} u\|_{L^{\infty}}+\|\partial^{4}_{y}u\|_{L^{\infty}})\nonumber\\
&\le C_u(\tau^{2}|\ln(\tau)|+h^2),\;\; {1}\leq{n}<{N}.
		\end{split}
	\end{equation*}
 On the other side, we know from Lemma  \ref{TnorderIEFD} 
\begin{equation*}
		\begin{split}
			|\widetilde{\xi}^{0}_{j,k}|=|{\xi}^{0}_{j,k}|&\leq
 C_{u}(\tau + h^{2}),\;\; {1}\leq{n}<{N}.
		\end{split}
	\end{equation*}
This completes the proof of Lemma \ref{Tnorder-bdf2}.

\end{appendix}

\bibliographystyle{siam}


\end{document}